\documentclass[12pt]{amsart}
\usepackage{mathrsfs}
\date{Nov.~22, 2010}

\usepackage{color}
\usepackage[all]{xy}
\usepackage{amssymb,amsmath}
\usepackage{graphicx}
\usepackage{setspace}

\calclayout
\newtheorem{dummy}{anything}[section]
\newtheorem{theorem}[dummy]{Theorem}
\newtheorem*{thma}{Theorem A}
\newtheorem*{thmb}{Theorem B}

\newtheorem{lemma}[dummy]{Lemma}
\newtheorem{proposition}[dummy]{Proposition}
\newtheorem{corollary}[dummy]{Corollary}

\theoremstyle{definition}%%Change Theoremstyle
\newtheorem{definition}[dummy]{Definition}
  \newtheorem{example}[dummy]{Example}
  
  \newtheorem{remark}[dummy]{Remark}

  \newtheorem*{acknowledgement}{Acknowledgement}

\newcommand
{\eqncount}{\setcounter{equation}{\value{dummy}}%
\addtocounter{dummy}{1}}
\newcommand{\cP}{\mathcal P}

\newcommand{\cU}{\mathscr U}
\newcommand{\cV}{\mathscr V}

\newcommand{\bZ}{\mathbb Z}

\newcommand{\bbH}{\mathbb H}

\newcommand{\bbR}{\mathbb R}

\DeclareMathOperator{\Aut}{Aut}
\DeclareMathOperator{\Image}{im}

\DeclareMathOperator{\rk}{rank}
\newcommand{\bigast}{\divideontimes}

\newcommand{\cy}[1]{\bZ/{#1}}

\newcommand{\vv}{\, | \,}

\newcommand\Fp{\bZ_p}

\newcommand{\Hlf}[1]{H^{l{\hskip-1pt}f, #1}}
\newcommand{\HLF}{H^{l{\hskip-1pt}f}}
\newcommand{\ch}{HC}
\DeclareMathOperator{\asdim}{asdim}
\newcommand{\C}{C}
\renewcommand{\L}{L}
\newcommand{\ZZ}{Y}
\begin{document}

\title{Coarse Geometry and P.~A.~Smith Theory}
\author{Ian Hambleton}

\address{\vbox{\hbox{Department of Mathematics \& Statistics, }
\hbox{McMaster University, Hamilton, ON L8S 4K1, Canada}}}
\email{hambleton@mcmaster.ca }

%\author{Lucian Savin}
%\address{29 Daniel Bram Drive
%Maple, ON L6A 0L4,  Canada}
%\email{savinnl@yahoo.com}

\author{Lucian Savin}
\address{BMO, 100 King Street West, Toronto, ON M5X 1A1,  Canada}
\email{savinnl@yahoo.com}

\thanks{Research partially supported by NSERC Discovery Grant A4000. The first author would like  to thank the Max Planck Institut f\"ur Mathematik in Bonn and the Topology Group at M\"unster for their hospitality and support while part of this work was done.}

\begin{abstract} 
We define a generalization of the fixed point set, called the bounded
fixed set, for a  group acting by isometries on a metric space. An analogue of the P.~A.~Smith theorem is proved for metric spaces of finite asymptotic dimension, which relates the
coarse homology of the bounded fixed set to the coarse homology of the total space.
\end{abstract}

\maketitle

\section{Introduction} One of the most important tools in transformation groups is 
P.~A.~Smith theory \cite{pasmith1}, \cite[Chap.~III]{bredon1}, which gives constraints on the homology of the fixed point set for actions of finite $p$-groups. 
For topological actions the fixed point sets may not be manifolds, but ``generalized manifolds" with complicated local topology (in the sense of Wilder \cite[Chap.~I.3]{borel-seminar}).  This means that an appropriate homology theory must be used to capture the essential features.

Smith theory in the generalized manifold setting, as developed in the 1960 classic ``Seminar on Transformation Groups" \cite{borel-seminar}, was used recently by Bridson and Vogtmann \cite{bridson-vogtmann1} to study the actions of  $\Aut(F_n)$, the automorphism group  of a free group, on acyclic homology manifolds and generalized homology $m$-spheres.

In this paper we provide a ``coarse homology" version of P.~A.~Smith theory  suitable for further applications in geometric group theory.  We study discrete groups of isometries of metric spaces, from the
perspective of ``large-scale" geometry introduced by M.~Gromov in \cite{gromov1993}. This subject is now known as coarse geometry. We introduce a
coarse generalization of the usual fixed set, called the \emph{bounded fixed set} (see Definition
\ref{bdd}).  It is defined when the coarse type of a certain sequence of approximate
fixed sets  stabilizes, even when the actual fixed set is empty. A group action is called \emph{tame} if the bounded
fixed set exists with respect to any subgroup. 

We say that a metric space is a (mod $p$) \emph{coarse homology $m$-sphere} if
it has the same (mod $p$) coarse homology as the Euclidean space $\bbR^m$.
The main application is:

\begin{thma} Let  $X$ be a proper
geodesic metric space with finite asymptotic dimension, which is a (mod $p$) coarse
homology $m$-sphere, for some prime $p$. Let $G$ be a finite  $p$-group with a tame action on $X$ by isometries. Then $X^G_{bd}$ is
a (mod $p$)  coarse homology $r$-sphere, for some $0\le r\le m$. If $p$ is odd, then $m-r$ is even.
\end{thma}

The coarse geometry of group actions extends to \emph{quasi-actions} on proper metric spaces (see Section \ref{sec: quasi}).  In particular, the bounded fixed set is a quasi-isometry invariant (see Proposition \ref{bdfix-same}).
The coarse analogues of the usual Smith theory inequalities are established in Theorem \ref{thm: pasmith}, and used to derive Theorem A in Section \ref{proof of thma}.  

Another well-known application of the classical P.~A.~Smith theory is  that a rank two group $G = \cy p \times \cy p$, for $p$ a prime, can not act freely on a finitistic mod $p$ homology $m$-sphere (see Bredon \cite[III.8.1]{bredon1}). 
In Theorem \ref{thm: semifree}, we give a coarse version of this result.

\begin{thmb} The group $G = \cy p \times \cy p$, for $p$ a prime, can not act tamely and semifreely at the large scale on a (mod $p$) coarse
homology $m$-sphere $X$, whenever
 $X$ is a proper geodesic metric space with finite asymptotic dimension, and
$X^G_{bd}$ is
a (mod $p$)  coarse homology $r$-sphere, for some $0\le r< m$.
\end{thmb}

We do not yet know complete necessary and sufficient conditions for
tameness of actions on a given metric space. 
Example \ref{theexample} shows that the sequence of approximate fixed sets does not
always stabilize. On the other hand, in Section \ref{sec: hyperbolic} we show that the action of any finite subgroup
of isometries of hyperbolic $n$-space, or more generally  any proper $\text{CAT}(0)$ space, is tame.  In Theorem \ref{thm: existence}, we show that a finite group action on a coarsely homogeneous metric space $X$ is tame (e.g.~if $X$ admits a compatible proper and cocompact discrete group of isometries).

\section{Coarse geometry}

Coarse geometry studies the properties of coarse spaces and coarse maps. We will
consider only the metric examples of coarse spaces. For the general definition of a
coarse space see Roe \cite{roe3} or Mitchener \cite{mitchener1}, \cite{mitchener2}.

\begin{definition}\label{def: coarse map}
Let $(X, d_X)$ and $(Y, d_Y)$ be metric spaces and $f \colon X\to Y$ a map, not necessarily
continuous.
\begin{itemize}
\item[(a)] The map $f$ is (metrically) \emph{proper} if the inverse image under $f$
of any bounded subset of $Y$ is a bounded subset of $X$.
\item[(b)] The map $f$ is \emph{bornologous} if for every $R>0$ there is $S_R>0$ such that
$$d_X(x_1,x_2)\le R \Rightarrow d_Y\bigl(f(x_1),f(x_2)\bigr)\le S_R,$$
for all $x_1,x_2\in X$.
\item[(c)] The map $f$ is \emph{coarse} if it is proper and bornologous.
\end{itemize}
\end{definition}

Two maps $f$, $f'$ from a set $X$ to a metric space $Y$ are said to be \emph{close} if
$d_Y\bigl(f(x),f'(x)\bigr)$ is bounded, uniformly in $X$.

\begin{definition} \label{coarse-equiv}
Two metric spaces $X$ and $Y$ are \emph{coarsely equivalent} if there exist coarse maps
$f \colon X\to Y$ and $f'\colon Y\to X$ such that $f\circ f'$ and $f'\circ f$ are close to the
identity maps on $Y$ and on $X$ respectively. The maps $f$ and $f'$ are called
\emph{coarse equivalences}.
\end{definition}
We remark that if $f\colon X \to Y$ and $g\colon Y \to Z$ are coarse equivalences, then the composite $g\circ f \colon X \to Z$ is also a coarse equivalence.
\begin{definition}\label{eff_prop}
Let $X$, $Y$ be metric spaces and $f \colon X\to Y$.
\begin{itemize}
\item[(a)] $f$ is called \emph{eventually Lipschitz} (or \emph{large-scale Lipschitz})
if there are positive constants $\L$ and $C$ such that
$$ d_Y\bigl(f(x_1),f(x_2)\bigr)\le \L\cdot d_X(x_1,x_2)+C, $$
for any $x_1, x_2\in X$.
\item[(b)] $f$ is called \emph{effectively proper} if for every $R>0$, there is
$S>0$ such that the inverse image under $f$ of each ball of radius $R$ in $Y$ is
contained in a ball of radius $S$ in $X$.
\end{itemize}
\end{definition}

\begin{lemma}\label{equiv-to-image}
Let $f \colon X\to Y$ be a coarse equivalence and $A\subseteq X$. Then $A$ and $f(A)$ (with
the induced metrics) are coarsely equivalent.
\end{lemma}

\begin{proof}
The restriction of $f$ to $A$ is a coarse map. For any
$y\in f(A)$ we choose $x\in f^{-1}(y)$ and define $\bar{f}(y)=x$. We obtain a map
$\bar{f}\colon f(A)\to A$ such that $f\circ\bar{f}=id_{f(A)}$ and $\bar{f}\circ f$ is close
to $id_A$. One can easily check that $\bar{f}$ is a coarse map.
\end{proof}

\begin{definition}
Let $X$ and $Y$ be two metric spaces. A map $f \colon X \to Y$ is an 
$(\L, \C)$-\emph{quasi-isometric embedding}
for the positive constants $\L$ and $\C$, if for any
$x_1,\ x_2 \in X$ we have:
\eqncount
\begin{equation}\label{quasi}
\frac{1}{\L} \cdot d_X(x_1, x_2)-\C\le d_Y\bigl(f(x_1), f(x_2)\bigr)
\le \L\cdot d_X(x_1, x_2) + \C.
\end{equation}
If $f(X)$ is also \emph{coarsely dense} in $Y$ (that is, if
 any point in $Y$ lies in the
$\C$-neighbourhood $N_{\C}(f(X))$ of $\Image  f$), then $f$ is called an 
$(\L, \C)$-\emph{quasi-isometry}, or just a \emph{quasi-isometry} for short. 
\end{definition}
\begin{remark} If $f\colon X \to Y$ and $g\colon Y \to Z$ are quasi-isometries, then the composite $g\circ f\colon X \to Z$ is also a quasi-isometry. 
\end{remark}
Any quasi-isometric embedding is a coarse map: the first part of the above
inequality shows that $f$ is proper and the second part shows that $f$ is eventually
Lipschitz, thus bornologous. The next result shows that any quasi-isometry is a coarse
equivalence.
\begin{proposition}[{\cite[p.~138]{bridson-haefliger1}}]
If $f \colon X \to Y$ is a quasi-isometry, then there exists a quasi-isometry
$f'\colon Y \to X$ such that $f\circ f'$ is close to the identity map on $Y$ and $f'\circ f$
is close to the identity map on $X$.
\end{proposition}

\begin{proof} Let $y\in Y$. If $y\in f(X)$, then choose $x\in X$ such that $f(x)=y$. If
$y\notin f(X)$, choose $\overline{x}\in X$ such that $d_Y\bigl(f(\overline x), y\bigr)\le C$.
Define $f'\colon Y \to X$ by:
$$f'(y)=\begin{cases}
x& \text{if }y\in f(X),\\
\overline{x} & \textrm{if }y \notin f(X).
\end{cases} $$

For any $x_1, x_2 \in f^{-1}(y)$ we have that $d_X(x_1, x_2) \le \L\cdot \C$
(from \ref{quasi}), so
$d_X\bigl(f'(f(x)), x\bigr) \le \L\cdot \C$ for any $x \in X$. Also, from the
definition of $f'$ we get that $d_Y\bigl(f(f'(y)), y\bigr) \le C$ for any
$y \in Y$. Therefore $f\circ f'$ and $f'\circ f$ are close to the identity maps on $Y$
and on $X$ respectively.

The image $f'(Y)$ is coarsely dense in $X$:
for any $x \in X$, let $y=f(x)\in Y$. Then
$d_X\bigl(x,f'(y)\bigr)= d_X\bigl(x,f'(f(x))\bigr)\le\L\cdot\C$.

Using (\ref{quasi}) and the triangle inequality, one can prove that $f'$ is a quasi-isometric embedding.
\end{proof}

For certain metric spaces the converse also holds.
\begin{proposition}[{\cite[1.10]{roe3}}]
If $X$ and $Y$ are length spaces, then any coarse equivalence $f \colon X\to Y$ is
a quasi-isometry.
\end{proposition}

\section{Quasi-actions}\label{sec: quasi}
Let $(X,d)$ be a metric space and $G$ be a discrete group. We say that $G$ \emph{acts coarsely} on $X$  (or $X$ is a \emph{coarse $G$-space}) if there are positive constants $\L$ and $\C$, and a map $\varphi\colon G\times  X \to X$ such that
\begin{enumerate}
\item For each $g\in G$, the map $x \mapsto \varphi(g,x):=g\cdot x$ is an $(\L, \C)$-quasi-isometry of $X$ (with $N_{\C}(g\cdot X) = X$).
\item The identity $e\in G$ acts as the identity on $X$, so $e\cdot x = x $ for all $x \in X$.
\item  For each $x\in X$ and each $g, h \in G$, $d(g(hx), (gh)x) \leq \C$.
\end{enumerate}
\begin{remark}
Sometimes the  condition (ii)  is omitted in the definition of coarse actions. Notice that any coarse action on $X$ in this more general sense is coarsely $G$-equivalent (via the identity map on $X$) to a coarse action on $X$ in which  $e\cdot x = x$, for all $x\in X$. 
\end{remark}
A coarse $G$-action is also called a \emph{quasi-action} of $G$ on $X$ (see \cite{kleiner-leeb1}). 
Sometimes we say that $G$ has an $(\L, \C)$-quasi-action on $X$ to specify the constants.
\begin{enumerate}
\item A coarse action of $G$ on $X$ is \emph{cobounded} if there exists $R>0$ such that for each $x\in X$, we have $N_R(G\cdot x) = X$.
\item A coarse action is \emph{proper} if for each $R>0$, there exists $M>0$ such that  for all $x, y \in X$, we have $\sharp\{g\in G\vv g\cdot N_R(x) \cap N_R(y) \neq \emptyset\} \leq M$.
\end{enumerate}

\begin{definition}
Let $X$ and $Y$ be coarse $G$-spaces.  A map $f \colon X\to Y$
is called \emph{coarsely $G$-equivariant} if there is a constant $N$ such that
$d_Y\bigl(gf(x), f(gx)\bigr)\le N$ for any $g\in G$ and $x\in X$.  We say that the actions are \emph{coarsely $G$-equivalent} if there exists  a quasi-isometry 
  $f\colon X \to Y$ which is a coarsely $G$-equivariant map.
\end{definition}

Coarsely $G$-equivalent $G$-actions are also called \emph{quasi-conjugate} in the literature. We also remark that the properties \emph{cobounded} or \emph{proper} for coarse actions are preserved by coarse $G$-equivalence. 

\begin{lemma}[Milnor-\u{S}varc]\label{milnor-svarc}
Let $(X,d)$ be a proper geodesic metric space. If $\Gamma$ is a discrete group with a proper, cobounded coarse action on $X$, then $\Gamma$ is  finitely-generated, and $X$ is quasi-isometric to the group $\Gamma$ with word metric.
\end{lemma}
\begin{proof} See Ghys and de la Harpe \cite[Proposition 10.9]{ghys-delaharpe2}. 
\end{proof}

If $f \colon X\to Y$ is a coarse $G$-equivalence, then the inverse quasi-isometry $f'\colon Y\to X$ (as in Definition \ref{coarse-equiv})
is also coarsely $G$-equivariant:  for any $y\in Y$, we have $d_Y\bigl(y, f(x)\bigr)\le M$,
for $x=f'(y)\in X$, by definition of $f'$. Then if $G$ has an $(\L', \C')$-quasi-action on $Y$, we have
\begin{gather*}
d_X\bigl(f'(gy),gf'(y)\bigr)\le d_X\bigl(f'(gy),f'(gf(x))\bigr)+
d_X\bigl(f'(gf(x)),f'(f(gx))\bigr)\\
+d_X\bigl(f'(f(gx)),gx\bigr)+d_X\bigl(gx,gf'(f(x))\bigr)+d_X\bigl(gf'(f(x)), gf'(y)\bigr)\\
\le (L+1)M + 2C + LS'_M + S'_{L'M + C'} + S'_N.
\end{gather*}
Here we have assumed that $d_X(f'(f(x)), x) \leq M$ for all $x \in X$, and denoted the constants for $f'$ from Definition \ref{def: coarse map} by $S'_N$, etc.

\begin{lemma} Let $f\colon X \to Y$ be a quasi-isometry, and suppose that $Y$ has a coarse action of  a discrete group $G$. Then $X$ admits a coarse action of $G$ so that $f$ is a coarse $G$-equivalence.
\end{lemma}
\begin{proof} 
The given coarse action of $G$  on $(Y,d)$ induces a map $G\times X \to X$, via the formula $g\cdot x := f'(g\cdot f(x))$, where $f'$ denotes a quasi-inverse for $f$. Since $f$ is a quasi-isometry, it follows that this formula defines a coarse action of $G$ on $(X,d)$.  The induced coarse action on $X$ is coarsely $G$-equivalent (by $f$) to the original coarse action on $Y$. 
\end{proof}
A recent result of Kleiner and Leeb \cite{kleiner-leeb1} shows that a coarse action is always coarsely equivalent to an isometric action.
\begin{theorem}[{Kleiner-Leeb \cite[Corollary 1.1]{kleiner-leeb1}}]\label{kleiner-leeb}If a discrete group $G$ has an $(\L, \C)$-quasi-action $\varphi\colon G \times X \to X$ on a metric space $(X, d)$, then $\varphi$ is $(\L, 3\C)$-quasi-conjugate to a canonically defined isometric $G$-action on a metric space $(Y,d)$.
\end{theorem}
\begin{remark} The proof of Kleiner and Leeb \cite[p.~1566]{kleiner-leeb1} shows that if $G$ is a finite group and $(X,d)$ is a proper metric space, then so is $(Y,d)$.
\end{remark}

\section{Bounded fixed sets}\label{sec: four}

Let $(X,d)$ be a metric space with an $(\L, \C)$-quasi-action of a discrete group $G$. For any $k\ge 0$, let
$$X^G_k =
\{x\in X\vv d(x, gx) \le k,\ \forall g\in G \}.$$ The sets
$\{X^G_k\}$ form an increasing family of subsets of $X$ and one
can ask if their coarse geometry type stabilizes. The following definition was given in \cite{savin1}.

\begin{definition}\label{bdd} We say that the \emph{bounded fixed set of a coarse action $(X, G)$ exists}, provided that there exists a subspace $Y \subseteq X$ such that 
 \begin{enumerate}
 \item $Y \subset X^G_{k_0}$ for some $k_0 > 0$, and 
 \item the inclusion map 
$i\colon  Y\to X^G_k$ is a coarse equivalence for all $k\ge k_0$.
\end{enumerate}
 In this case we  write $Y = X^G_{bd}$. If
the coarse type of the subspaces $\{X^G_k\}$ does not stabilize, we say that $X^G_{bd}$ does
not exist.
\end{definition}

In general, $X^G_k$ can be empty for all $k$ (and in this case we have $X^G_{bd} = \emptyset$). For example, take
$G= \bZ$ acting on $X= \bbR$ by translations.
However, if $G$ is finite, then the sets $X^G_k$ are always
nonempty for large $k$. In fact, we have that
$$X=\bigcup_{k \ge 0} X^G_k.$$

\begin{remark}
The inclusion $i\colon X^G_k\to X^G_r$ ($r\ge k$) preserves the metric, so it is a quasi-isometric
embedding. If $X^G_k$ is coarsely dense in $X^G_r$, then $i$ is a quasi-isometry,
therefore a coarse equivalence. Note that the subspaces are coarsely $G$-invariant, in the sense that $G\cdot X^G_k \subseteq X^G_{r}$, where $r = Lk + 4C$.
\end{remark}

The orbit space $X/G$ of an isometric action has
 a natural (pseudo)-metric $d^*$ induced by the standard projection
$p\colon X\to X/G$ from the metric on $X$:

$$d^*\bigl(p(x),p(y)\bigr)=\inf_{g\in G}d(x,gy).$$
If $(X, G)$ is a quasi-action, we can define $X/G := X'/G$, where $(X', G)$ is any isometric action coarsely equivalent to $(X,G)$, as provided by Theorem \ref{kleiner-leeb}. This construction and the induced metric $d^*$ are both well-defined up to quasi-isometry.

\begin{definition}\label{def: coarsely ineffective}
A coarse $G$-action is called \emph{coarsely ineffective} if the map $p\colon  X\to X/G$ is a
coarse equivalence.
\end{definition}
If $G$ acts coarsely ineffectively on $X$, there is a coarse equivalence
$h\colon X/G\to X$. For any $x\in X$ and $g\in G$, we have that $p(x)=p(gx)$, so
$h\bigl(p(x)\bigr)=h\bigl(p(gx)\bigr)$. Also, from the definition of a coarse
equivalence, there is a constant $C$ so that $d\bigl(h\bigl(p(x)\bigr),x\bigr)\le C$
for any $x\in X$.Then
\begin{gather*}
d(gx,x)\le d\bigl(gx,h\bigl(p(gx)\bigr)\bigr)+d\bigl(h\bigl(p(x)\bigr),x\bigr)\le2C,
\end{gather*}
so $X^G_{bd}$ exists and is coarsely equivalent to $X$.

The converse also holds: if $X^G_{bd}$ exists and is coarsely equivalent to $X$, then
the action is coarsely ineffective. Choose $k>0$ so that $X^G_k$ is coarsely
dense in $X$. For any $x\in X$, there is $x'\in X^G_k$ such that $d(x,x')\le C$. Then
\begin{gather*}
d(gx,x)\le d(gx,gx')+d(gx',x')+d(x',x)\le 2C+k
\end{gather*}
It follows that
$$d(x,y)\le d(x,gy)+d(gy,y)\le d^*\bigl(p(x),p(y)\bigr)+2C+k$$
since given $\epsilon >0$, we can pick $g\in G$ such that $d(x,gy) \leq d^*\bigl(p(x),p(y)\bigr)+ \epsilon$.
Also, from the definition of $d^*$, we have that $d^*\bigl(p(x),p(y)\bigr)\le d(x,y)$.
The map $p$ is, obviously, surjective. Therefore, $p$ is a quasi-isometry, so it is a coarse
equivalence.

\begin{definition}
Let $(X,d)$ be a metric space with a coarse $G$-action. The coarse action of $G$ is called \emph{tame} if $X^H_{bd}$ exists for all
subgroups $H$ in $G$.
\end{definition}

It would be interesting to find a geometrical condition on $X$ which would
guarantee that the action of any finite subgroup of quasi-isometries of $X$ is tame.
We first point out that \emph{tameness} of the action is a coarse invariant.

\begin{proposition}\label{bdfix-same} Suppose that $(X, d)$ and $(Y,d)$ are 
coarse $G$-spaces.
If $f \colon X\to Y$ is a coarse $G$-equivalence
and $X^G_{bd}$ exists, then
$Y^G_{bd}$ exists and it is coarsely equivalent to $X^G_{bd}$.
\end{proposition}

\begin{proof}
The existence of $X^G_{bd}$ implies that there is some $k_0$ such that the inclusion
$i\colon X^G_{k_1}\to X^G_{k_2}$ is a coarse equivalence, for any $k_2\ge k_1\ge k_0$. For any
$x\in X^G_k$ we have
$$d_Y\bigl(f(x),gf(x)\bigr)\le d_Y\bigl(f(x),f(gx)\bigr)+
d_Y\bigl(f(gx),gf(x)\bigr)\le S_k+N,$$
thus $f(X^G_k)\subseteq Y^G_l$ for $l\ge S_k+N$. Similarly, $f'(Y^G_l)\subseteq X^G_r$,
for some $r>0$. Then, we have
$$ f'\bigl(f(X^G_k)\bigr)\subseteq f'(Y^G_l)\subseteq X^G_r, $$
for $r>k\ge k_0$. $X^G_k$ is coarsely dense in $X^G_r$, so for any $x\in X^G_r$, there is
$x'\in X^G_k$ such that $d_X(x,x')\le C$, for some constant $C$. Then
$f'\bigl(f(X^G_k)\bigr)$ is also coarsely dense in $X^G_r$: for any $x\in X^G_r$ we have:
$$d_X\bigl(x,f'(f(x'))\bigr)\le d_X(x,x')+d_X\bigl(x',f'(f(x'))\bigr)\le C+M. $$
It follows that $f'(Y^G_l)$ is also coarsely dense in $X^G_r$, thus they
are coarsely equivalent. Lemma \ref{equiv-to-image} implies that $Y^G_l$ is coarsely
equivalent to $X^G_{bd}$ for any $l$ greater than some value $l_0$.

To finish the proof, we need to
show that the inclusion $Y^G_{l_1}\to Y^G_{l_2}$ is a coarse equivalence, for any
$l_2\ge l_1\ge l_0$. We have that $f'(Y^G_{l_1})\subseteq f'(Y^G_{l_2})\subseteq X^G_r$
for some $r>0$ and that $f'(Y^G_{l_1})$ is coarsely dense in $X^G_r$. This means that
$f'(Y^G_{l_1})$ is coarsely dense in $f'(Y^G_{l_2})$. So, for any $y\in Y^G_{l_2}$
there is a $y'\in Y^G_{l_1}$ such that $d_X\bigl(f'(y),f'(y')\bigr)\le C$. Then
\begin{gather*}
d_Y(y,y')\le d_Y\bigl(y, f(f'(y))\bigr)+d_Y\bigl(f(f'(y)), f(f'(y'))\bigr)\\
+d_Y\bigl(y',f(f'(y'))\bigr)\le 2M+S_C,
\end{gather*}
which completes the proof.
\end{proof}

After this result, and Theorem \ref{kleiner-leeb}, to study the tameness of a coarse $G$-action we can assume that $G$ is acting by isometries on $(X,d)$. 
In the rest of the paper we will also assume that $G$ is a \emph{finite group}, and that $(X,d)$ is a proper metric space.

The next example shows that assuming finite asymptotic dimension (see Section \ref{dim-asym}, or Roe \cite[\S 9]{roe3}) is not sufficient to ensure that $X^G_{bd}$ exists.

\begin{example}[Non-existence]\label{theexample}
Let $G=\bZ/2$. We will construct a space $X\subset\bbR^3$, consisting of
infinitely many ``goalposts'', on which $G$ acts freely and
the coarse type of $X^G_k$ does not stabilize.

\begin{figure}[!htp]
\begin{center}
\includegraphics[height=5cm]{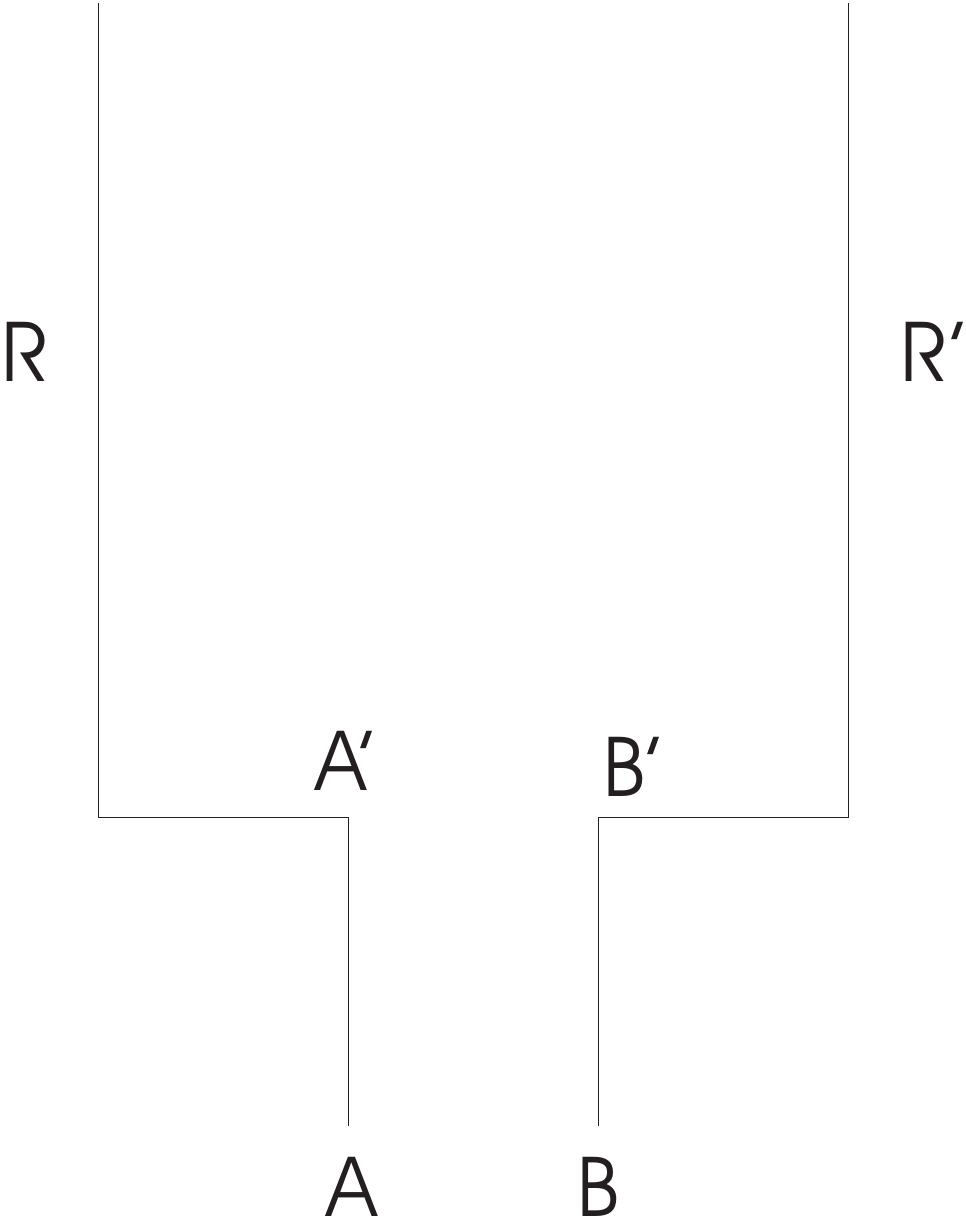}
\caption{A generic set $Y_n$}
\end{center}
\end{figure}

Start with two parallel lines, $Y_0=\{x=\frac{1}{2},z=0\}\cup\{x=-\frac{1}{2},z=0\}$.
Let $Y_n$ be the set shown in Figure 1 below, where the distances $AA'$, $BB'$ and $AB$
are equal to $1$ and the vertical rays $R$ and $R'$ are at distance $n$ from each other.

For any $n\ge1$, consider the points $(\frac{1}{2},n,0)$ and $(-\frac{1}{2},n,0)$
in $Y_0$ and identify these points to the
points $A$ and $B$ of a copy of $Y_n$ such that $Y_n$ makes an angle
of $\frac{n\pi}{2n+2}$ with the $xy$-plane (see Figure 2 below).

\begin{figure}[!htp]
\begin{center}
\includegraphics[height=5cm]{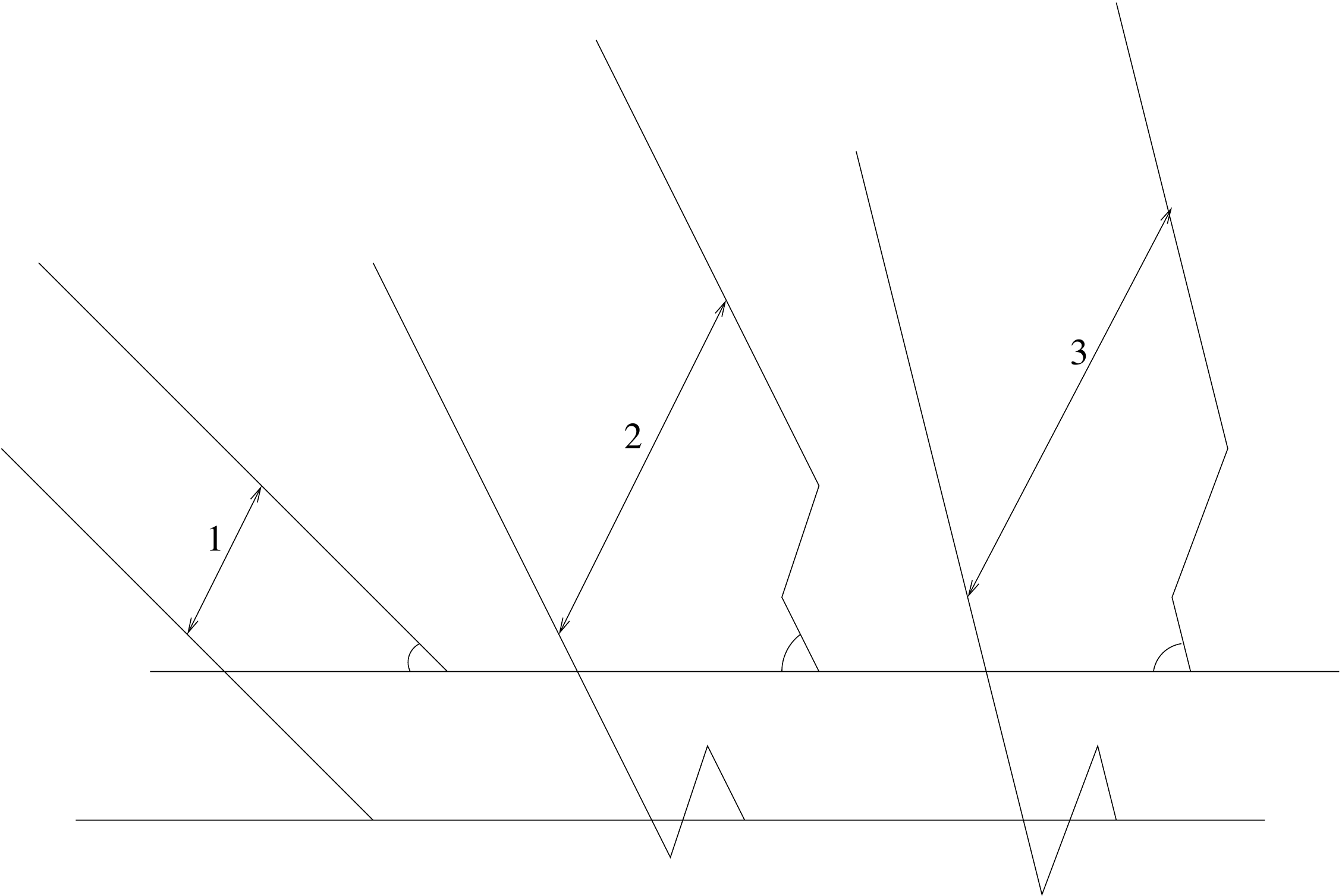}
\caption{The set $X$}
\end{center}
\end{figure}

We obtain a space $X=\bigcup Y_n$ on which $\bZ/2$ acts freely, by interchanging
the two branches. The asymptotic dimension of $X$ (with the induced metric) is at most 3, 
because it is a subset of $\bbR^3$ (see \cite[9.10]{roe3}).

One can see that the coarse type of the sets $X^G_k$ does not stabilize: for any integer
$n$, we have that $Y_n\subset X^G_n$, but $Y_n$ contains points arbitrarily far away from
$X^G_r$, for any $r<n$. So the coarse type of $X^G_k$ changes when $k$ takes integer values,
therefore it does not stabilize.
\end{example}

Here is a condition which ensures that that the bounded fixed set exists. 

\begin{theorem}\label{thm: existence} Let $(X,d)$ be a proper geodesic metric space. If $(X,d)$ admits a proper and cobounded coarse action by a discrete group $\Gamma$, then the bounded fixed set $X^G_{bd}$ exists for any finite subgroup $G \subset \Gamma$.
\end{theorem}
\begin{proof} By the Milnor-\u{S}varc Theorem \ref{milnor-svarc}, we may assume that $(X,d) = (|\Gamma|, d)$ is just the group $\Gamma$ with the (left invariant) word metric. The identity element $e\in \Gamma$ will be taken as a base-point, and the action of an element $g\in G$ on $X$ will be denoted $\gamma \mapsto g\cdot\gamma$, for $\gamma \in \Gamma$. 

Let $H = C_\Gamma(G) = \{ z\in \Gamma\vv gz = zg, \ \forall g\in G\}$ 
 denote the centralizer subgroup  for $G$ in $\Gamma$, and let $Y = |C_\Gamma(G)|\subset X$ denote the subspace of $X$ consisting of the group elements in the centralizer.
Since $d(z, g\cdot z) = d(e, z^{-1}(g\cdot z)) = d(e,g\cdot e)$, for all $z\in H$ and all $g\in G$, we see that $Y \subset X^G_k$ as soon as $k > \max \{d(e,g\cdot e) \vv g \in G\}$. We also observe that $H$ acts on $X^G_k$, for any $k >0$, defined by the formula $x \mapsto zx$, for $x\in X^G_k$ and $z\in H$, and $H\cdot X^G_k \subseteq X^G_{k}$.

Now suppose that $k>0$ is large enough so that $Y \subset X^G_k \neq \emptyset$. Let $S: =\{ y_1, y_2, \dots, y_t\}$ denote the distinct elements in the ball $B(e, k) \subset \Gamma $ of radius $k$ around the identity element $e\in \Gamma$. Since $d(x, g\cdot x) = d(e, x^{-1}(g\cdot x)) < k$ for any $x\in X^G_k$ and any $g\in G$, we see that each such element $x^{-1}(g\cdot x)$ must equal one of the $y_i$.
We fix an ordering of the elements of $G$ and obtain a map
$$\varphi\colon X^G_k \to \cP(S)$$ defined by
$x\mapsto \{x^{-1}(g\cdot x)\vv g \in G\}$ from $X^G_k$  to the finite subsets $\cP(S)$ of $S$. Since 
$\varphi(zx) =\varphi(x)$, for all $z\in H$ and all $x\in X^G_k$
it follows that $X^G_k$ is contained in the union of finitely many $H$-orbits. Pick elements $\{x_1, x_2, \dots, x_m\}$ in $X^G_k$ representing the distinct  $H$-orbits. Then any $x \in X^G_k$ can be expressed as $x = zx_i$, for some $z\in H$ and $1\leq i\leq m$. But then $d(z, x) = d(z, zx_i) = d(e, x_i)$. Therefore the inclusion $Y \subset X^G_k$ is coarsely dense (with maximum distance $N = \max \| x_i\|$, $1\leq i\leq m$).
\end{proof}
\begin{remark} We have actually shown that $X^G_{bd}$ is quasi-isometric to the subspace $Y = |C_\Gamma(G)|$ defined in the proof. 
\end{remark}

Here are some examples of tame actions:

\begin{example}[Euclidean space]
Let $X=\bbR^n$ with the Euclidean metric and $G$ be a finite group which
acts on $\bbR^n$ by isometries.  Then $X^G_{bd} = X^G$ is the linear subspace of $\bbR^n$ fixed by $G$.
\end{example}
\begin{example} (Semi-direct products) The example which inspired the definition of the bounded fixed set was a semi-direct product
$\Gamma=\bZ^n\rtimes_{\alpha}G$, given by an integral representation $\alpha\colon G\to GL_n(\bZ)$. In this case, $X = (|\Gamma|,d)$ is just the group $\Gamma$ equipped with the word metric, and $G$ acts by left multiplication. Then $X^G_{bd}$ is coarsely equivalent to the fixed sub-representation of $\bbR^n$ induced by the conjugation action $\alpha$ of $G$ on the normal subgroup $\bZ^n$. 
\end{example}

\section{Hyperbolic space}\label{sec: hyperbolic}

In this section we will show that the action of any finite subgroup of
$\text{Isom}(\bbH^n)$ on $\bbH^n$ is tame, where $\bbH^n$ is
the hyperbolic $n$-space and $\text{Isom}(\bbH^n)$ is its group of isometries.
We will use the Poincar\'e model, in which the points of hyperbolic
$n$-space are represented by the points of the open unit ball $B^n$ in $\bbR^n$.
The geodesic lines are the intersection of $B^n$ with those Euclidean lines and circles
which are orthogonal to the boundary of $B^n$. An advantage of this model is that the
angle between two geodesics issuing from the same point is the Euclidean angle between them.

\begin{definition}
Let $X$ be a metric space and consider two geodesic rays $c,c'\colon [0,\infty)\to X$.
We say that $c$ and $c'$ are \emph{asymptotic} if there is a
constant $K$ such that $d\bigl(c(t),c'(t)\bigr)\le K$ for any $t\ge0$.
\end{definition}

One can easily check that this is an equivalence relation. The set of equivalence
classes is called the \emph{boundary} of $X$ and is denoted by $\partial X$. The
points of $\partial X$ are called \emph{points at infinity}. The equivalence
class of a geodesic ray $c$ will be denoted $c(\infty)$.

Notice that the images of two asymptotic geodesic rays under any isometry of $X$ are
again asymptotic geodesic rays. So, if $G$ is a group which acts on $X$ by isometries,
there is an induced $G$-action on $\partial X$.

\begin{proposition}[\cite{bridson-haefliger1}]
Let $c\colon [0,\infty)\to\bbH^n$ be a geodesic ray issuing from $x$ (i.e. $c(0)=x$).
Then, for any $x'\in\bbH^n$, there is a unique geodesic ray which issues from
$x'$ and is asymptotic to $c$.
\end{proposition}

\begin{proof} See Bridson-Haefliger \cite{bridson-haefliger1}, Chapter II.8, Proposition 8.2.
\end{proof}

Let $X$ be a complete $\text{CAT}(0)$ space (see \cite{bridson-haefliger1}, Chapter II.1,
Definition 1.1). There is a topology on $\overline X=X\cup\partial X$, called
the \emph{cone topology}, such that the subspace topology of $X$ is the original metric
topology. A neighbourhood basis for the points at infinity has the following
form: given a geodesic ray $c$ and positive constants $\varepsilon>0$ and $r>0$, then
\begin{gather*}
U(c,r,\varepsilon)=\{x\in X\vv d\bigl(x,c(0)\bigr)>r, d\bigl(p_r(x),c(r)\bigr)
<\varepsilon\}\\
\cup \{\xi\in\partial X\vv d\bigl(p_r(\xi),c(r)\bigr)<\varepsilon\}
\end{gather*}
where $p_r\colon \overline X\to B\bigl(c(0),r\bigr)$ is the projection of $\overline X$ onto
the closed ball $B\bigl(c(0),r\bigr)$ defined by:

$$ p_r(x)=\begin{cases}
x& \text{if }x\in B\bigl(c(0),r\bigr),\\
c'(r) & \textrm{if }x\in\overline X\backslash B\bigl(c(0),r\bigr),
\end{cases} $$
where $c'$ is the geodesic ray issuing from $c(0)$ and passing through $x$ (if $x\in X$)
or representing $x$ (if $x\in\partial X$).

If $X$ is the Poincar\'e model of the hyperbolic $n$-space, it is known that
$\partial X$ is homeomorphic to the $(n-1)$-sphere in $\bbR^n$.

Let $G$ be a finite subgroup of $\text{Isom}(\bbH^n)$. Then the
fixed point set $\bbH^G$ is nonempty (see \cite{bridson-haefliger1}, Chapter II.2,
Corollary 2.8 (1)). One can show that if $\bbH^G$ is bounded, then the action on
the boundary is free: if some point at infinity is fixed by $G$, let $c$
be the geodesic ray issuing from a fixed point $x\in\bbH^G$, so that
$c(\infty)$ is fixed. Then $c$ and $gc$ are asymptotic geodesic rays issuing
from the same point, thus $c(t)=gc(t)$, for all $t>0$, therefore $\bbH^G$ contains
a geodesic ray, which contradicts our assumption.

The converse also holds: if the action on $\partial\bbH^n$ is free, then
$\bbH^G$ is bounded. We will prove this claim by contradiction. Suppose that
$\bbH^G$ is unbounded. Then fix a point $x_0\in\bbH^G$ and, for
any $m>0$, there is $x_m\in\bbH^G$ so that $d(x_0,x_m)\ge m$. Consider the
geodesic rays $c_m$ issuing from $x_0$ and passing through $x_m$. Then $c_m(\infty)$
is a sequence in $\partial\bbH^n$, thus it has a convergent subsequence. Since
$c_m$ passes through the fixed points $x_0$ and $x_m$, it follows that $gc_m=c_m$
for all $g$ and $m$.

To simplify the notations, we assume that $\{c_m(\infty)\}$ converges to a point
$\xi\in\partial\bbH^n$. Since $G$ acts freely on the boundary, $g\xi\ne\xi$ if
$g$ is not the identity in $G$. Choose $c$ the geodesic ray issuing from $x_0$ so
that $c(\infty)=\xi$. Then, for any $\varepsilon>0$, there is $T_\varepsilon>0$
such that $d\bigl(c(t),gc(t)\bigr)>\varepsilon$ for all $t\ge T_\varepsilon$.

Choose $\varepsilon=1$, so $d\bigl(c(t),gc(t)\bigr)>1$ for all $t$ bigger or equal to
some $T_1$. Fix $\tau>T_1$ and choose $m$ so
that $c_m(\infty)\in U(c,\tau,\frac{1}{3})$. Then
\begin{gather*}
d\bigl(c(\tau),gc(\tau)\bigr)\le d\bigl(c(\tau),c_m(\tau)\bigr)+
d\bigl(c_m(\tau),gc_m(\tau)\bigr)\\
+d\bigl(gc_m(\tau),gc(\tau)\bigr)\le\frac{1}{3}+\frac{1}{3}<1,
\end{gather*}
which is the desired contradiction. One can slightly adjust this argument to prove
the following:

\begin{proposition}
$\bbH^G$ is coarsely dense in $\bbH^G_k$, for any $k>0$.
\end{proposition}

\begin{proof}
Suppose that, for some $k$, $\bbH^G$ is not coarsely dense in $\bbH^G_k$.
Thus, for any $m\ge1$, there is $x_m\in\bbH^G_k$ so that $d(x_m,\bbH^G)\ge m$.
Fix $x_0\in\bbH^G$ and let $c_m$ be the geodesic ray issuing from $x_0$
and passing through $x_m$. Then $c_m(\infty)$ is a sequence in $\partial\bbH^n$,
thus it has a convergent subsequence.

Without loss of generality, we can assume that $\{c_m(\infty)\}$ converges to a point
$\xi\in\partial\bbH^n$ which is not fixed by $G$ (if it were, it would be
represented by a geodesic ray contained in $\bbH^G$, which is clearly false). Let
$c$ be the geodesic ray issuing from $x_0$ so that $c(\infty)=\xi$. Then, for any
$\varepsilon>0$, there is $T_\varepsilon>0$ such that
$d\bigl(c(t),gc(t)\bigr)>\varepsilon$ for all $t\ge T_\varepsilon$.

Choose $\varepsilon=k+1$, so $d\bigl(c(t),gc(t)\bigr)>k+1$ for all $t$ bigger or equal
to some $T_1$. Fix $\tau>T_1$ and choose $m$ so
that $x_m=c_m\bigl(d(x_0,x_m)\bigr)\in U(c,\tau,\frac{1}{3})$. Then
$d\bigl(c_m(\tau),gc_m(\tau)\bigr)\le d(x_m,gx_m)\le k$ and
\begin{gather*}
d\bigl(c(\tau),gc(\tau)\bigr)\le d\bigl(c(\tau),c_m(\tau)\bigr)+
d\bigl(c_m(\tau),gc_m(\tau)\bigr)\\
+d\bigl(gc_m(\tau),gc(\tau)\bigr)\le\frac{1}{3}+k+\frac{1}{3}<k+1,
\end{gather*}
which is a contradiction.
\end{proof}

It follows that the $G$-action on $\bbH^n$ is tame and the $G$-action on
$\partial\bbH^n$ is free iff $\bbH^G$ is bounded and a point at
infinity is fixed iff it is a limit point of
$\bbH^G$.

\begin{remark}
The above argument works for any proper $\text{CAT}(0)$ space $X$ because it has compact
boundary (see \cite[p.~264]{bridson-haefliger1}). We conclude that the action of any finite subgroup $G$ of $\text{Isom}(X)$ is tame.
\end{remark}

\section{Asymptotic dimension}\label{dim-asym}

The notion of asymptotic dimension was introduced by Gromov in \cite[p.~29]{gromov1993} 
and is a coarse geometry analogue to the topological covering dimension of a compact metric
space. This section describes the basic properties of spaces with finite asymptotic dimension (for more information, see  Dranishnikov \cite[\S 4]{dranishnivkov_2000}, Roe \cite[Chap.~3]{roe3},  or Bell-Dranishnikov \cite{bell-dranishnikov_2008}).

\begin{definition}
We say that $X$ has \emph{asymptotic dimension} $\le l$ if for each $r>0$ the space $X$
can be decomposed into a union of $l+1$ subsets
$$ X=\bigcup_{k=0}^l X_k, $$
where each $X_k$ is \emph{r-disconnected}: each $X_k$ is a disjoint union of sets of
uniformly bounded diameter, and these sets are at least $r$ apart from each other
$\bigl($where $dist(A_1, A_2)=\inf\{d(x_1, x_2)\vv \forall x_1\in A_1,\ x_2\in A_2\}\bigr)$.
\end{definition}

If $Y\subseteq X$, then $\asdim Y\le\asdim  X$. Using the definition,
one can check that $\asdim  \bbR\le1$ and $\asdim  \bbR^2\le2$.
In fact, one can show that $\asdim \bbR^n=n$ (that means
$\asdim \bbR^n\le n$ but $\asdim \bbR^n\not\le n-1$).

Before stating the main result of this section, we need to review some facts about
simplicial complexes (see Bredon \cite{bredon1}, Chapter III, Section 1).

An \emph{abstract simplicial complex} is a set $K$,
whose elements are called vertices, together with a
collection of finite nonempty subsets of $K$, called
simplices such that:
\begin{itemize}
\item[(a)] every vertex is contained in some simplex,
\item[(b)] every nonempty subset of a simplex is a simplex.
\end{itemize}
The \emph{dimension} of a simplicial complex $K$ is $n$ if
$K$ contains an $n$-simplex, but no $(n+1)$-simplices
or $\infty$ if $K$ contains $n$-simplices for any $n\ge 0$.

A \emph{simplicial map} $f \colon K_1\to K_2$ is a function
from the vertices of $K_1$ to the vertices of $K_2$ such
that the image of any simplex of $K_1$ is a simplex
of $K_2$. Two simplicial maps $f, f'\colon K_1\to K_2$ are
\emph{contiguous} if, for any simplex $s\in K_1$, $f(s)$
and $f'(s)$ belong to a common simplex of $K_2$.
Two contiguous maps induce the same map in homology.

Let $H$ be the Hilbert space $\ell^2(K)$. Define a map
$K\to H$ by sending any vertex $v\in K$ to the corresponding
unit vector $e_v\in H$. For any simplex $s=(v_0,\dots ,v_n)$
of $K$, its \emph{geometric realization} (or the
\emph{closed simplex}) $|s|$ is the subset of $H$
consisting of all convex combinations of
$e_{v_0},\dots ,e_{v_n}$ (all linear combinations
$\Sigma\lambda_ve_v$ with positive coefficients such
that $\Sigma\lambda_v=1$). The \emph{geometric realization}
(or the \emph{polyhedron}) $|K|$ of $K$ is the union
of the geometric realizations of its simplices. The
induced metric $d_B$ on $|K|$ is called the metric
of \emph{barycentric coordinates}. The \emph{length metric} 
(see \cite[p.~33]{bridson-haefliger1})
associated to the metric of barycentric coordinates is called the
\emph{intrinsic metric} on $K$.

A simplicial complex is \emph{locally finite} if any
vertex belongs to finitely many simplices.
The \emph{barycentric subdivision} of a simplicial
complex $K$ is the simplicial complex $K'$ whose vertices
are the simplices of $K$ and whose simplices are the
sets $(s_0,\dots, s_n )$ of vertices of $K'$ (simplices
of $K$), such that, after reordering,
$$ s_0\subset s_1\subset\dots\subset s_n$$ (each $s_i$
is a face of $s_{i+1}$). There is a canonical homeomorphism
$ |K'|\approx |K|, $
(see \cite{eilenberg-steenrod1}, Chapter II, Lemma 6.2).

Let $G$ be a finite group which acts on $K$ such that each
transformation is a simplicial map. Such an action is
called a \emph{simplicial action} and $K$ with a simplicial
action is called a \emph{simplicial $G$-complex}.

\begin{definition}\label{regular complex}
A simplicial $G$-complex $K$ is called \emph{regular}
if it satisfies the following condition:
\begin{itemize}
\item[(A)] for any subgroup $H\subseteq G$, if
$h_0, \dots ,h_n$ are in $H$ and $(v_0, \dots ,v_n)$
and $(h_0v_0, \dots ,h_nv_n)$ are simplices of $K$,
then there is an element $h\in H$ such that $hv_i=h_iv_i$,
for all $i$.
\end{itemize}
\end{definition}

\noindent The condition (A) implies the following equivalent conditions:
\begin{itemize}
\item[(B\hphantom{$^\prime$})] for $v\in K$ and $g\in G$, if $v$ and $gv$ belong
to the same simplex, then $v=gv$;
\item[(B$^\prime$)] for any simplex $s$ of $K$ and $g\in G$,
$g$ fixes every vertex in $s\cap g(s)$.
\end{itemize}

\begin{proposition}[\cite{bredon1}] If $K$ is a simplicial $G$-complex, then the
induced action on the barycentric subdivision $K'$ satisfies
\textup{(B)}. If the action on $K$ satisfies \textup{(B)}, then the action on
$K'$ is regular.
\end{proposition}

\begin{proof}
See Bredon \cite{bredon1}, Section III, Proposition 1.1.
\end{proof}

Let $f \colon X\to |K|$ be a map from $X$ to the polyhedron of a simplicial complex $K$. We say
that $f$ is \emph{uniformly cobounded} if there is a uniform finite bound for the diameter
of the inverse image under $f$ of the star of any vertex of $K$.

The \emph{degree} of a covering $\cU $ is the maximum number of members of
$\cU $ with nonempty intersection.

The following result gives different characterizations for spaces with
finite asymptotic dimension. The definition of a ``coarsening system" is given in the next section (see Definition \ref{coarsening}). 
\begin{theorem}[\cite{roe3}]\label{finit-cech}
Let $X$ be a proper metric space. Then the following are equivalent:
\begin{itemize}
\item[(a)] $X$ has asymptotic dimension $\le l$;
\item[(b)] $X$ admits a coarsening system consisting of coverings of degree $\le l+1$;
\item[(c)] For any $\varepsilon >0$ there is an $\varepsilon$-Lipschitz and uniformly
cobounded map from $X$ to an $l$-dimensional polyhedron equipped with the metric of
barycentric coordinates.
\end{itemize}
 If $X$ is a geodesic space, these conditions are also equivalent to
\begin{itemize}
\item[(d)] For any $\varepsilon >0$ there is an $\varepsilon$-Lipschitz and
effectively proper
map from $X$ to an $l$-dimensional polyhedron equipped with the intrinsic metric.
\end{itemize}
\end{theorem}

\begin{proof}
See Roe \cite[Theorem 9.9]{roe3}.
\end{proof}

\section{Coarsening systems and Coarse homology}\label{relativ}

In this section we will recall the definition of  coarse homology \cite{roe1}, \cite{roe3}, \cite{higson-roe1}. All metric spaces considered are proper.

\begin{definition}\label{def: uniform}
A  covering $\cU$
of $X$ with the property that any $U\! \in \cU$ is relatively compact and
any bounded subset of $X$ intersects only finitely many sets in $\cU $
is called a \emph{uniform covering} of $X$.
\end{definition}

It follows that any uniform covering is locally finite.
Let $\cU $ be a uniform covering of $X$ and let $\{\varphi_U\vv U\in\cU \}$
be a partition of unity subordinate to $\cU $ (which exists because, by Stone's
theorem, any metric space is paracompact).
Let $K(\cU )$ be the nerve of $\cU $ and regard
its geometrical realization $|K(\cU )|$ as a subspace of the Hilbert space
$H=\ell^2(\cU )$. Define a map $\Phi\colon X\to|K(\cU )|$ by
$$ \Phi(x)=\sum_U \varphi_U(x)e_U, $$
where $e_U$ is the unit vector of $H$ associated to $U\in\cU $.
The maps $\{\varphi_U\}$ are continuous, so the map $\Phi$ is also continuous:
if $x_n\to x$ in $X$, then $\varphi_U(x_n)\to\varphi_U(x)$, for any $U$,
thus $\Phi(x_n)\to\Phi(x)$.

\begin{definition}[{Roe \cite[3.13]{roe1}}]\label{coarsening}
A \emph{coarsening system} of $X$ (or an
\emph{anti-\v Cech system} of $X$) is a sequence
of uniform coverings $\{\cU_n\}$ of $X$ for which there exists an increasing
sequence of real numbers $R_n \to \infty$ such that for all $n$,
\begin{itemize}
\item[(a)] each set $U\!\in \cU_n$ has diameter less than or equal to
$R_n$,
\item[(b)] the covering $\cU_{n+1}$ has a Lebesgue number greater than or
equal to $R_n$.
\end{itemize}
\end{definition}

Any proper metric space admits a coarsening system $\{\cU_n\}$ (see Roe \cite[Lemma 3.15]{roe1}). 
It follows that the covering $\cU_n$ is a refinement of $\cU_{n+1}$,
so there exist refinement projections $\beta_n\colon \cU_n\to\cU_{n+1}$,
for any $n$. These maps are called \emph{coarsening maps}. From now on, we
will include a choice of such maps as part of a coarsening system. 

If $\cU$ is a covering of $X$, one can define a simplicial complex
called the nerve of $\cU$ and denoted by $K(\cU)$. The vertices are the
 sets $U$ of the covering and the simplices are finite non-empty subsets
$\{U_0,\dots ,U_n\}$ of $\cU$ with non-empty intersection.
If the nerve of a covering is locally finite, then the covering is
locally finite, but the converse is not true. However, if the covering is uniform,
then its nerve is a locally finite simplicial complex.

The coarsening maps $\beta_n$ induce proper simplicial maps
$K(\cU_n)\to K(\cU_{n+1})$.
A different choice of the coarsening maps induces contiguous maps.

The \emph{locally finite} homology of a simplicial complex $K$ is defined
using chains that are infinite, locally finite formal linear combinations
of oriented simplices of $K$.
For example, consider the 1-dimensional simplicial complex $K$ whose set of vertices
is $\bZ$ and the simplices are of the form $\{n,n+1\}$ for any $n\in\bZ$
(it follows that $|K|=\bbR$). The 0-dimensional locally finite homology
group is trivial (each vertex is the boundary of an infinite 1-chain consisting
of all the simplices to the left of it), whereas the 1-dimensional locally finite
homology group is non-trivial (the sum of all the 1-simplices is a generator). In fact,
for any field $F$, it is true that 
\eqncount
\begin{equation}\label{euclidean space}
\HLF_q(\bbR^n;F)=\begin{cases}
F& \text{if }q=n,\\
0& \text{otherwise}.
\end{cases} 
\end{equation}

\begin{definition}[{Roe \cite[\S 5.5]{roe3}}]
Let $\{\cU_n\}$ be a coarsening system of $X$ and
let $\HLF_*$ be the locally finite homology theory. Then the
\emph{coarse homology} of $X$ is given by:
$$ \ch_*(X)=\varinjlim \HLF_*\bigl(K(\cU_n)\bigr).$$
\end{definition}

Let $A\subset X$ and denote by $K(\cU_n|A)$ the subcomplex of $K(\cU_n)$
consisting of those simplices $(U_0,\dots , U_q)$ such that
$$U_0\cap \dots \cap U_q \cap A\neq \emptyset.$$

One can check that the complex $K(\cU_n|A)$ is isomorphic to the nerve of
$A\cap\cU_n=\{A\cap U\vv U\in\cU_n\ \text{and}\ A\cap U\ne\emptyset\}$.
The coverings $\{A\cap\cU_n\}$ form a coarsening system of $A$ (they have
increasing Lebesgue number and sets with uniformly bounded diameter), thus
$$\ch_*(A)=\varinjlim \HLF_*\bigl(K(\cU_n|A)\bigr).$$

\begin{definition}
If $A\subset X$ and $\{\cU_n\}$ is a coarsening system of $X$, then the
\emph{relative coarse homology} is defined by
$$\ch_*(X,A)=\varinjlim \HLF_*\bigl(K(\cU_n), K(\cU_n|A)\bigr).$$
\end{definition}

Two different coarsening systems of $X$ will give
rise to canonically isomorphic coarse homology groups, therefore the coarse homology is
independent of the coarsening system (see \cite[p.~229]{higson-roe1}).

\begin{proposition}[\cite{higson-roe1}]
Two close maps $f,f'\colon X\to Y$ induce the same map on $\ch_*$.
\end{proposition}

\begin{proof}
See Higson-Roe \cite[Proposition 2.2]{higson-roe1}.
\end{proof}

\begin{remark}
An easy consequence of this proposition is that each coarse equivalence induces an
isomorphism in the coarse homology.
\end{remark}

\section{Regular $G$-coarsening systems}

In this section, we will define the appropriate equivariant version of a coarsening system, for $(X,d)$ a  coarse $G$-space (see Section \ref{sec: quasi}) and $G$ a finite group. A \emph{regular $G$-coarsening system} is the anti-\v Cech analogue of the systems of regular $G$-coverings constructed in Bredon \cite[III.6]{bredon1}. 
 Theorem \ref{regular coarsening system} shows the existence of such systems of  coverings under certain assumptions.

\begin{definition}\label{def: regular coarsening system}
 Let $X$ be a coarse $G$-space, for $G$ a finite group. We say that a  covering $\cU$ of $X$ is  is
\emph{$G$-invariant} or that $\cU$ is a $G$-\emph{covering}, provided that
\begin{enumerate}
\item $g\,\cU=\{ g\, U\vv U\!\in\cU\} = \cU$, for all $g\in G$.
\item $g_1(g_2U) = (g_1g_2)U$, for all $g_1, g_2 \in G$ and all $U \in \cU$.
\end{enumerate}
\end{definition}

If $G$ acts by isometries on $X$, then  $g\,\cU$ is also a covering of $X$ and the second condition is automatic. In that case, $\cU$ is a $G$-covering if and only if
$g\,\cU=\cU$ for all $g\in G$.  

If $X$ is a coarse $G$-space, there is a natural induced $G$-action on the
nerve  of a $G$-invariant covering,
$$ g(U_0,\dots,U_n)=(gU_0,\dots,gU_n), $$ 
making $K(\cU)$  a simplicial $G$-complex. We will always be working with uniform coverings (see Definition \ref{def: uniform}), so that the nerves $K(\cU)$ will also be locally-finite.

\begin{definition}
Let $\cU$ be a uniform $G$-invariant  covering of $X$. Then $\cU$ is a
\emph{regular} $G$-covering if its nerve $K(\cU)$ is a regular $G$-complex. A
\emph{regular $G$-coarsening system} for $X$ is coarsening system  $\{\cU_n\}$, such that  $\cU_n$ is a regular $G$-covering, for any $n\in\mathbb{N}$.
\end{definition}

We observe that the existence of a regular $G$-coarsening system is a coarse invariant.
\begin{lemma}\label{lem: coarsening}
 Suppose that $(X, d)$ and $(Y,d)$ are 
coarse $G$-spaces.
If $f \colon X\to Y$ is a coarse $G$-equivalence
and $Y$ admits a regular $G$-coarsening system, then so does $X$.
\end{lemma}
\begin{proof} If $\{\cU_n\}$ is a regular $G$-coarsening system for $Y$, then for each of the coverings $\cU_n$ we let $\cV_n = \{ f^{-1}(U) \vv U \in \cU_n\}$. From the metric properties of $f$ it is clear that each covering $\cV_n$ is a uniform $G$-covering, and that its nerve $K(\cV_n)$ is a regular $G$-complex. By passing to a subsequence of the coverings $\{\cV_n\}$, if necessary, we can obtain an anti-\v Cech system, and therefore a regular $G$-coarsening system for $X$.
\end{proof}

\medskip
We will now show that  any proper geodesic metric space $X$ with finite asymptotic dimension admits a
regular coarsening system, whose $G$-invariant coverings have nerves with uniformly bounded dimension. In the rest of this section, we assume that $G$ acts by isometries on $X$.

\begin{lemma}
Let $\cU$ be a uniform covering of $X$, and $G$ a finite group of isomtries of $X$. Then the collection
$$\widetilde{\cU}=\bigsqcup_{g\in G}g\,\cU=
\{g \, U\vv g\in G,\: U\!\in\cU\}$$
 is a $G$-invariant uniform covering of $X$.
\end{lemma}

\begin{proof}
It is clear that the sets in $\widetilde{\cU}$
form a $G$-invariant covering by construction. Since $G$ acts by isometries,
$\text{diam}\,gU=\text{diam}\,U$. Thus the sets in $\widetilde{\cU}$ are
bounded subsets of a proper metric space, and therefore  relatively compact.

Let $A\subset X$ be a bounded set. Suppose that $A$ intersects infinitely many
sets from $\widetilde{\cU}$. Then, since $G$ is finite, there exists
$g\in G$ such that $A\cap gU_k \neq\emptyset$ for infinitely many sets
$U_k \in\cU$. Therefore the bounded set $g^{-1}(A)$ intersects infinitely
many sets from $\cU$, which contradicts that $\cU$ is uniform.
\end{proof}

If $\{\cU_n\}$ is a coarsening system
of $X$, then $\{\widetilde{\cU}_n\}$ is a sequence of
uniform $G$-coverings of $X$. We will show that the coverings
 $\{\widetilde{\cU}_n\}$  also form a coarsening system for $X$. If  $\asdim  X\le l$, we may assume that $\dim K(\cU_n)\leq l$ for any $n$ (see Theorem \ref{finit-cech}).

\begin{lemma}\label{inv.coars.syst}
If $\{\cU_n\}$ is a coarsening system
of $X$, the sequence $\{\widetilde{\cU}_n\}$ is a coarsening
system for $X$. Moreover, if $\dim  K(\cU_n)\le l$, then the covering $\{\widetilde{\cU}_n\}$ has degree  $\leq p(l+1)$, where $p=|G|$.
\end{lemma}

\begin{proof}
We check the properties from the definition of a coarsening
system. For any $n\in\mathbb{N}$, $g\in G$ and $U\in
\cU_n$, we have that
$\mathrm{diam}\,gU=\mathrm{diam}\,U\le R_n$, so the sets in
$\widetilde{\cU}_n$ have uniformly bounded diameter. The
covering $\cU_n$ is a refinement of
$\widetilde{\cU}_n$, so the Lebesgue number of
$\widetilde{\cU}_n$ is at least $R_{n-1}$. Therefore $\{\widetilde{\cU}_n\}$ is a coarsening system.

Assume that $\dim K(\cU_n)\le l$,  and suppose  that there are at least $p(l+1)+1$ sets from $\widetilde{\cU}_n$ with
nonempty intersection. Then, for some $g\in G$, there are at least $l+2$ sets
$\{ U_1,\dots,U_q \}$ of $\cU_n$ such that $\bigcap gU_k \neq\emptyset$. Therefore 
$\bigcap U_k \neq\emptyset$, so the nerve $K(\cU_n)$ has dimension at least
$l+1$, which is impossible.
\end{proof}

So far, we have constructed a coarsening system consisting of $G$-invariant coverings whose nerves
have uniformly bounded dimension. To obtain a regular $G$-coarsening system, we need to
introduce some new notions (see \cite[p.~133]{bredon1}).

Let $\cU$ be a locally finite $G$-invariant
covering of $X$ and let $\phi =\{ \varphi_U\vv U\in\cU \}$ be a partition of
unity subordinate to $\cU$. Then $\phi$ is called a \emph{G-partition of unity}
if $\varphi_{gU}(gx)=\varphi_U(x)$, for all $g$, $x$ and $U$. If
$f=\{f_U\vv U\in\cU\}$ is any partition of unity subordinate to the
$G$-invariant covering $\cU$, we define a $G$-partition $\phi$ by putting
$$ \varphi_U(x)=\frac{1}{|G|}\sum_g f_{gU}(gx). $$
If each $f_U$ is $\varepsilon$-Lipschitz, then $\varphi_U$ is also
$\varepsilon$-Lipschitz:
\eqncount
\begin{equation}\label{epsilon-lip}
\big|\varphi_U(x)-\varphi_U(y)\big|\le\frac{1}{|G|}\sum_g\big|f_{gU}(gx)-f_{gU}(gy)\big|\le
\varepsilon d(x,y).
\end{equation}
Recall that the map $\Phi\colon X\to |K(\cU)|$ defined by:
$$\Phi(x)=\sum_U \varphi_U(x)\, e_U$$
is continuous.
If $\cU$ is a $G$-invariant open covering of $X$ and
$\{ \varphi_U \ | \ U\in\cU \}$
is a $G$-partition of unity subordinate to $\cU$, then $\Phi$ is also equivariant:
\begin{gather*}
\Phi(gx)=\sum_U \varphi_U(gx)\, e_U=\sum_U \varphi_{gU}(gx)\, e_{gU} \\
=\sum_U \varphi_U(x)\, e_{gU}=g\sum_U \varphi_U(x)\, e_U=g\Phi(x).
\end{gather*}

The following observation will be used in the proof of the main result of this section.

\begin{remark}
If $f \colon X\to Y$ is a $\varepsilon$-Lipschitz map between two metric spaces
$X$ and $Y$ and if $\cV =\{ V_\alpha \}$ is an open covering of $Y$ with Lebesgue
number $r$, then
$\cU=f^{-1}\cV=\{f^{-1}(V)\vv V\in\cV\}$ is an open
covering of $X$ with Lebesgue number at least $r/ \varepsilon$.
\end{remark}

\begin{theorem}\label{regular coarsening system}
For any proper geodesic metric space $X$ with asymptotic dimension $\leq l$,  and for any finite group $G$ which acts on $X$ by isometries, there is a
regular $G$-coarsening system of $X$ consisting
of $G$-invariant coverings of degree at most $p(l+1)$, where $p=|G|$.
\end{theorem}

\begin{proof}
Let $\{\widetilde{\cU}_n\}$ be a $G$-invariant coarsening system of $X$, given
by Lemma \ref{inv.coars.syst}. For any $\varepsilon >0$ there exists $n\in \mathbb{N}$
such that $\Phi\colon X\to |K(\widetilde{\cU}_n)|$ is $\varepsilon$-Lipschitz and
effectively proper map, where $|K(\widetilde{\cU}_n)|$ is equipped with
the intrinsic metric (from Theorem \ref{finit-cech}).
Moreover, we can assume that $\Phi$ is $G$-equivariant (see equation \ref{epsilon-lip}).

Let $L_n$ be the second barycentric subdivision of $K(\widetilde{\cU}_n)$ and
regard the polyhedra $|L_n|=|K(\widetilde{\cU}_n)|$ as equal. Then $L_n$ is a
locally finite regular $G$-complex. Denote by $\Phi^{-1}L_n$ the
covering of $X$ by inverse images of open vertex stars of $|L_n|$. We will show that a
subsequence of $\{\Phi^{-1}L_n\}$ form a coarsening system with the desired properties.

Since $\Phi$ is continuous, $\Phi^{-1}L_n$ is an open covering
whose sets are uniformly bounded (because $\Phi$ is effectively proper).
Its nerve is isomorphic to $L_n$ (since $\Phi$ is $G$-equivariant)
which is a regular $G$-complex, thus $\Phi^{-1}L_n$ is a regular $G$-covering.

Let $A$ be a bounded subset of $X$. Then $\Phi(A)$ is a bounded subset of
the proper metric space $|K(\widetilde{\cU}_n)|$.
Then the set
$$P=\{x\in |K(\widetilde{\cU}_n)|\text{ such that }
d\bigl(x,\Phi(A)\bigr)\le\sqrt{2}\}$$
is also bounded. For any vertex $v\in K(\widetilde{\cU}_n)$, its open star
$st(v)$ intersects $\Phi(A)$ if and only if $v\in P$. Since $cl(P)$ is compact, $P$
contains finitely many vertices, so $\Phi(A)$ intersects only finitely many
open vertex stars of $L_n$. Thus, $\Phi^{-1}L_n$ is a uniform covering of $X$.

The covering of $|L_n|$ by open vertex stars has a positive Lebesgue number
$r$ (depending only on the dimension of $L_n$), then $\Phi^{-1}L_n$ has a Lebesgue
number at least $r/ \varepsilon$. Since we can choose $\varepsilon$
arbitrarily small, we can construct regular,
uniform coverings of $X$ with arbitrarily large Lebesgue number.
\end{proof}

\section{The coarse homology of a bounded fixed set}

We assume that $(X,d)$ is a proper metric space, equipped with a coarse $G$-action of a finite group $G$. In this section, we show how a regular $G$-coarsening system for $X$ (see Definition \ref{def: regular coarsening system}) can be used to compute the coarse homology of the bounded fixed set, whenever the bounded fixed set exists. The answer is given in Corollary \ref{rhomap}.

Let $\{\cU_n\}$ be a regular $G$-coarsening system for $X$. By Lemma \ref{lem: coarsening}, the existence of a regular $G$-coarsening system is a coarse invariant. Hence by Theorem \ref{kleiner-leeb}, we will assume that $G$ acts by isometries on $X$.
Recall that
$$X^G_r =\{x\in X\vv d(x, gx) \le r,\ \forall g\in G \}$$
and suppose that $X^G_{bd}$ exists  (see Definition \ref{bdd}).

Consider the nerve
$K(\cU_n)$ and recall that $K(\cU_n|X^G_r)$ denotes the subcomplex of
$K(\cU_n)$ consisting of those simplices $(U_0,\dots , U_q)$ such that
$$U_0\cap \dots \cap U_q \cap X^G_r \neq \emptyset.$$
$K(\cU_n|X^G_r)$ is isomorphic to the nerve of
$X^G_r\cap\cU_n=\{X^G_r\cap U\vv U\in\cU_n\}$. These coverings form
a coarsening system of $X^G_r$, so we have that
$$\ch_*(X^G_r)= \lim_{n\to\infty} \HLF_*\bigl(K(\cU_n|X^G_r)\bigr).$$
We \emph{define} the coarse homology of $X^G_{bd}$ to be
$$\ch_*(X^G_{bd})=\lim_{r\to\infty} \ch_*(X^G_r)=
\lim_{r\to\infty}\lim_{n\to\infty} \HLF_*\bigl(K(\cU_n|X^G_r)\bigr).$$

 We would like to compare $\HLF_*\bigl(K(\cU_n|X^G_r)\bigr)$
to $\HLF_*\bigl(K(\cU_n)^G\bigr)$, where $K(\cU_n)^G$ denotes
the subcomplex of $K(\cU_n)$ spanned by those vertices $U\in\cU_n$ invariant
under the action of $G$. In order to achieve this, we will use their intersection
$K(\cU_n)^G_r = K(\cU_n|X^G_r)\cap K(\cU_n)^G$. Explicitly:
$$K(\cU_n)^G_r=\{(U_0,\dots , U_q)\in K(\cU_n)^G \,|\
U_0\cap\dots\cap U_q \cap X^G_r \ne\emptyset\}.$$

We will need the following algebraic lemma from Bredon \cite{bredon1}.

\begin{lemma}\label{direct limit diagram}
Let $D$ be a directed set and let $\{A_i, f_{ij}\}$
and $\{B_i, g_{ij}\}$ be direct systems of abelian groups based on $D$. Let
$\{\theta_i\colon A_i\to B_i\}$ be a homomorphism of directed systems. Assume that
for each index $i$ there is an index $j>i$ and a homomorphism
$h_{ij}\colon B_i\to A_j$ such that the diagram
$$\xymatrix{
 {A_i} \ar[r]^{{\theta}_i} \ar[d]_{f_{ij}} &
   {B_i} \ar[d]^{g_{ij}} \ar[dl]|-{h_{ij}} \\
 {A_j} \ar[r]_{{\theta}_j} & {B_j} } $$
commutes. Then the induced map
$\theta\colon \lim A_i\to \lim B_i$
is an isomorphism.
\end{lemma}

\begin{proof} See Bredon \cite{bredon1}, Chapter III, Lemma 6.4.
\end{proof}

\begin{proposition}With the above notations, we have that
$$\lim_{n\to\infty}\HLF_*\bigl(K(\cU_n)^G_r\bigr)\cong \lim_{n\to\infty}
\HLF_*\bigl(K(\cU_n|X^G_r)\bigr).$$
\end{proposition}

\begin{proof}
Take $(U_0,\dots ,U_q)\in K(\cU_n|X^G_r)$ and pick
$x\in U_0\cap\dots\cap U_q\cap X^G_r$. Then
diam$(U_i\cup Gx)\le R_n+r$ for any $i$ from $0$ to $q$.
Choose $m>n$ such that $R_m\ge R_n+r$, then
$U_i\cup Gx\subset V_i$, for some $V_i\in K(\cU_m|X^G_r)$.
Since $gV_i\cap V_i\ne\emptyset$ and $K(\cU_m|X^G_r)$ is $G$-regular,
it follows that $V_i$ is $G$-invariant, which means that
$(V_0,\dots ,V_q)\in K(\cU_m)^G_r$.
In this way we defined a map $K(\cU_n|X^G_r)\to K(\cU_m)^G_r$
which makes the following diagram commute:
$$\xymatrix{
{K(\cU_n)^G_r} \ar[r] \ar[d] & {K(\cU_n| X^G_r)}
\ar@{-->}[dl] \ar[d] \\
{K(\cU_m)^G_r}  \ar[r] & {K(\cU_m| X^G_r)} }
$$
where the horizontal maps are inclusions and the vertical ones are
the coarsening maps associated to a coarsening system. The proof follows from
Lemma \ref{direct limit diagram}.
\end{proof}
\begin{remark}\label{rem: fixed covering}
The proof just given shows that for any $r >0$, and any $n >0$, there exists an $m >n$ so that
$\cU_m^G$ is a covering of $X^G_r$. To see this, note that $X^G_r \cap \cU_n$ is a covering and each $U \in \cU_n$ with $U \cap X^G_r \neq \emptyset$ is shown to be contained in $\cU_m^G$, where $m >n$ is chosen so that
$R_m \geq R_n + r$.
\end{remark}

\begin{corollary}\quad
$ \ch_*(X^G_{bd})\cong \lim_{r\to\infty}\lim_{n\to\infty}
\HLF_*\bigl(K(\cU_n)^G_r\bigr).$
\end{corollary}

We have that for any $r>0$, $K(\cU_n)^G_r$ is a subcomplex of
$K(\cU_n)^G$ and, for any $r'>r$, $K(\cU_n)^G_r$ is a
subcomplex of $K(\cU_n)^G_{r'}$, so we have a commutative diagram:

$$\xymatrix{
  {\lim\limits_{n\to\infty} \HLF_*\bigl(K(\cU_n)^G_r\bigr)} \ar[r] \ar[d] &
  {\lim\limits_{n\to\infty} \HLF_*\bigl(K(\cU_n)^G\bigr)} \\
  {\lim\limits_{n\to\infty} \HLF_*\bigl(K(\cU_n)^G_{r'}\bigr)} \ar[ur] & }
$$
From the definition of the direct limit we get a map
$$\rho\colon\lim_{r\to\infty}\lim_{n\to\infty}\HLF_*\bigl(K(\cU_n)^G_r\bigr)
\to \lim_{n\to\infty}\HLF_*\bigl(K(\cU_n)^G\bigr).$$

\begin{theorem}
The map $\rho$ is an isomorphism.
\end{theorem}

\begin{proof} We will show first that $\rho$ is surjective.
Let $w=(w_p, w_{p+1}, \dots )$ be an element in
$\lim\limits_{n\to\infty}\HLF_*\bigl(K(\cU_n)^G\bigr)$, then
$w_p$ is a cycle in $C_*\bigl(K(\cU_p)^G\bigr)$ of the form:
$$ w_p=\sum_i\lambda_i\sigma_i $$
where $\sigma_i=(U_0,\dots ,U_q)$ is a simplex in
$K(\cU_p)^G$. Then $U_0\cap\dots\cap U_q$ is an invariant set
with diameter less than $R_p$, so the diameter of the orbit of any
point in this set is bounded by $R_p$. Thus
$\sigma_i\in K(\cU_n)^G_r$ for any $r>R_p$.

Since $r$ is independent of $\sigma_i$, the same one will work for
all $\sigma_i$ in $w_p$ which means that $w_p$ is a cycle in
$C^{lf}_*\bigl(K(\cU_p)^G_r\bigr)$ which defines an element in the group
$\lim\limits_{r\to\infty}\lim\limits_{n\to\infty}\HLF_*\bigl(K(\cU_n)^G_r\bigr)$. By
construction, $\rho$ maps this element to $w$, so $\rho$ is surjective.

To show that $\rho$ is injective, let $z=(z^{l})_{\scriptscriptstyle l>0}\in\ker (\rho)$.
There exists some $r$ such that $z^r$ is mapped to the zero
element in $\lim\limits_{n\to\infty}\HLF_*\bigl(K(\cU_n)^G\bigr)$. Let
$z^r=(z^r_p, z^r_{p+1},\dots )$ and suppose that $z^r_p$ is a boundary in
$C_*\bigl(K(\cU_p)^G\bigr)$, so $z^r_p=\partial w_p$ where $w_p$ is a
chain in $C_{*+1}\bigl(K(\cU_p)^G\bigr)$. An argument similar to the one used above
shows that $w_p\in C_{*+1}\bigl(K(\cU_p)^G_s\bigr)$ for some $s\ge r$.
Since the map
$$\HLF_*\bigl(K(\cU_p)^G_r\bigr)\to \HLF_*\bigl(K(\cU_p)^G_s\bigr)$$
is induced by the inclusion
$K(\cU_p)^G_r\to K(\cU_p)^G_s$, it follows that
$z^s_p=0$ in $\HLF_*\bigl(K(\cU_p)^G_s\bigr)$, which implies that $z=0$ in
$\lim\limits_{r\to\infty}\lim\limits_{n\to\infty}\HLF_*\bigl(K(\cU_n)^G_r\bigr)$.
\end{proof}

Therefore, we obtain a characterization of the coarse homology of the bounded fixed set.
\begin{corollary}\label{rhomap} Let $X$ be a coarse $G$-space, for $G$ a finite group. Assume that $X$ admits a regular $G$-coarsening system $\{\cU_n\}$.  Then
$$\ch_*(X^G_{bd})=\varinjlim \HLF_*\bigl(K(\cU_n)^G\bigr)$$
provided that the bounded fixed set exists.
\end{corollary}

\section{The coarse homology of the orbit space}

Let ($X,d$) be a proper metric space and let $G$ a finite
group which acts on $X$ by isometries.
Define a map $d^\bigast\colon X/G\times X/G\to\bbR$ by
$$d^\bigast(\bar{x},\bar{y})=\min_{g\in G}d(x,gy),$$
with $p(x)=\bar{x}$ and $p(y)=\bar{y}$, where $p\colon X\to X/G$ is the canonical projection.
This is independent of the choice of $x\in p^{-1}(\bar{x})$ and $y\in p^{-1}(\bar{y})$
because each map $g$ is an isometry.

\begin{lemma}
The map $d^\bigast$ is a metric and ($X/G,d^\bigast$) is a proper metric space.
\end{lemma}

\begin{proof}
It is obvious that $d^\bigast$ is symmetric and takes nonnegative values. Suppose that
$d^\bigast(\bar{x},\bar{y})=0$. Then $x=gy$, for some $g\in G$ and we have that $\bar{x}=p(x)=p(gy)=\bar{y}$.

Let $\bar{x},\bar{y},\bar{z}\in X/G$. Then
$d^\bigast(\bar{x},\bar{z})\le d(x,gz)\le d(x,hy)+d(hy,gz),$
for any $g,h\in G$. Since $G$ is finite, $d^\bigast(\bar{x},\bar{y})=d(x,ky)$ for some $k\in G$.
It follows that $d^\bigast(\bar{x},\bar{z})\le d^\bigast(\bar{x},\bar{y})+d(ky,gz)$, for any $g\in G$.
If we choose $g$ so that $d(ky,gz)=d(y,k^{-1}gz)=d^\bigast(\bar{y},\bar{z})$, we obtain the triangle inequality:
$$d^\bigast(\bar{x},\bar{z})\le d^\bigast(\bar{x},\bar{y})+d^\bigast(\bar{y},\bar{z}). $$

For the second part, notice that $p$ is continuous with respect to the metric topologies on $X$ and $X/G$
(if $d(x_n,x)\to 0$, then
$d^\bigast\bigl(p(x_n),p(x)\bigr)\to 0$ as well, since $d^\bigast\bigl(p(x),p(y)\bigl)\le d(x,y)$, for any $x,y\in X$).
We will denote by $B(?,r)$ an open $r$-ball in either metric space and by $\overline{B}(?,r)$ its closure.
We want to show that
$p\bigl(\overline{B}(x,r)\bigr)=\overline{B}\bigl(p(x),r\bigr)$ for any $r>0$. For later reference, we will
show first that $p\bigl(B(x,r)\bigr)=B\bigl(p(x),r\bigr)$, which implies that $p$ is an open map:
if $U\subset X$ is open, then
$$ U=\bigcup_{x\in U} B(x,\varepsilon_x), $$
so $p(U)=\bigcup B\bigl(p(x),\varepsilon_x\bigr)$ is open in $X/G$.

Choose $\bar{y}\in p\bigl(B(x,r)\bigr)$, so there is $y\in B(x,r)$ with $p(y)=\bar{y}$. Since $d(x,y)<r$,
then $d^\bigast(\bar{x},\bar{y})<r$, where $\bar{x}=p(x)$, so $\bar{y}\in B(\bar{x},r)$.

Start with $\bar{y}\in B(\bar{x},r)$, so $d^\bigast(\bar{x},\bar{y})<r$, which means that, for some $g\in G$,
$d(x,gy)=d^\bigast(\bar{x},\bar{y})<r$ or $gy\in B(x,r)$, thus $\bar{y}=p(gy)\in p\bigl(B(x,r)\bigr)$.

Now let $\bar{y}\in p\bigl(\overline{B}(x,r)\bigr)$, so there exists $y\in \overline{B}(x,r)$
with $p(y)=\bar{y}$. Then, for any
$\varepsilon>0$, we have that $B(y,\varepsilon)\cap B(x,r)\neq\emptyset$. This implies that
$B(\bar{y},\varepsilon)\cap B(\bar{x},r)\not=\emptyset$, thus $\bar{y}\in \overline{B}(\bar{x},r)$.

If $\bar{y}\in \overline{B}(\bar{x},r)$, then for any $\varepsilon>0$, we have
$B(\bar{y},\varepsilon)\cap B(\bar{x},r)\ne\emptyset$. For any $\varepsilon=\frac{1}{n}$, there is
$\bar{z}_n\in B(\bar{y},\frac{1}{n})\cap B(\bar{x},r)$, so there exist $g_n,h_n\in G$ such that $d(x, g_nz_n)<r$
and $d(g_nz_n,h_ny)<1/n$. Since $G$ is finite, there is some $h\in G$ such that $h=h_n$ for
infinitely many $n$.

For any $\delta>0$, there is such an index $n$ with $\delta\ge\frac{1}{n}$, therefore
$B(hy,\delta)\cap B(x,r)\ne\emptyset$ which implies that $hy\in \overline{B}(x,r)$, thus
$\bar{y}=p(hy)\in p\bigl(\overline{B}(x,r)\bigr)$.

Let $\overline{A}\subset X/G$ be closed and bounded, so $\overline{A}\subset \overline{B}(\bar{x},r)$, for
some $\bar{x}\in X/G$ and some finite $r>0$. If $A=p^{-1}(\overline{A})\cap \overline{B}(x,r)$ (for some
$x\in p^{-1}(\bar{x})$), then $A$ is closed (since $p$ is continuous) and bounded in $X$. Thus $A$ is compact,
so $p(A)=\overline{A}$ is also compact, which completes the proof.
\end{proof}

\begin{remark}
In general, $\overline{B}(x,r)\varsubsetneq \{y\in X\vv d(x,y)\le
r\}$, for example, if $d$ is the discrete metric.
\end{remark}

Let $\cU$ be a regular $G$-covering of $X$,  and for each $U \in \cU$ let
$$G(U) = \bigcup \{ gU \vv {g\in G}\}$$
denote the union of the orbit of $U$ in $X$.
Let $\cU^\bigast$ denote the
covering of $X/G$ by the sets $U^\bigast=G(U)/G$ and indexed by $\cU/G$. 
If $G(U)=G(V)$ but $U\ne gV$ for any $g$, then
$U^\bigast$ and $V^\bigast$ are regarded as different elements of the covering.

\begin{lemma}[\cite{bredon1}]
If $\cU$ is a regular $G$-covering of $X$, then the assignment
$\{gU\vv g\in G\}\mapsto U^\bigast$ gives an isomorphism of the simplicial complexes
$$K(\cU)/G\to K(\cU^\bigast).$$
\end{lemma}

\begin{proof}
See Bredon \cite{bredon1}, Chapter III, Proposition 6.2.
\end{proof}

Suppose that $\cU$ and $\cV$ are regular $G$-coverings of $X$ such that
$\cV$ is a refinement of $\cU$. Then there is a refinement
projection $\pi\colon \cV\to\cU$ that is \emph{equivariant}; that is,
$V\subset \pi(V)$ and $\pi(gV)=g\pi(V)$. When $\pi$ is equivariant, the simplicial map
$\bar{\pi}\colon K(\cV)\to K(\cU)$ is also equivariant.

As an immediate consequence, given a regular $G$-coarsening system $\{\cU_n\}$,
we can assume that all the coarsening maps $\beta_n\colon \cU_n\to\cU_{n+1}$
are equivariant.

\begin{proposition}
Let $\{\cU_n\}$ be a regular $G$-coarsening system for $X$. Then $\{\cU^\bigast_n\}$
is a coarsening system of $X/G$.
\end{proposition}

\begin{proof}
Let $p\colon X\to X/G$ be the canonical projection and notice that $U^\bigast=G(U)/G=p(U)$.
It follows that any $U^\bigast\in\cU_n^\bigast$ is open (since $p$ is an open map) and
$\text{diam}\,U^\bigast\le\text{diam}\,U$ (see the definition
of the metric $d^\bigast$ on $X/G$ at the beginning of this section).

Let $\bar{A}\subset X/G$ with diam $\bar{A}\le k$. There exists some point $\bar{x}\in X/G$
such that $\bar{A}\subseteq B(\bar{x},k)$, the open ball of radius $k$ around $\bar{x}$.
Let $x\in X$ such that $p(x)=\bar{x}$. Then $p\bigl(B(x,r)\bigr)=B(\bar{x},r)$, so there
exists $A\subseteq B(x,r)$ with $p(A)=\bar{A}$. It follows that $\bar{A}$ cannot intersect
infinitely many sets from $\cU_n^\bigast$ because that would imply that one of
the translates of $A$ intersects infinitely many sets from $\cU_n$ (since
$G$ is finite and $\cU_n$ is regular), which contradicts that
$\cU_n$ is uniform. Thus $\cU_n^\bigast$ is a uniform covering.

One can also see that the Lebesgue number of $\cU_n^\bigast$ is at least half of the
Lebesgue number of $\cU_n$: if $\bar{A}\subset X/G$ with
diam $\bar{A}\le R_{n-1}/2$, then
$\bar{A}\subseteq B(\bar{x},\frac{R_{n-1}}{2})$, for some $\bar{x}\in X/G$.
As above, there is a set $A\subseteq B(x,\frac{R_{n-1}}{2})$ with $p(A)=\bar{A}$ and
$p(x)=\bar{x}$. Then diam $A\le R_{n-1}$, so $A\subset U$ for some $U\in\cU_n$,
thus $\bar{A}\subset U^\bigast$.
\end{proof}

From the last two results we have the formula:
\begin{corollary}\label{orbit space}
\quad $\ch_*(X/G)=\varinjlim \HLF_*\bigl(K(\cU_n)/G\bigr).$
\end{corollary}

\begin{remark}\label{rem: orbit}
 If the bounded fixed set $X^G_{bd}$ exists, then there exists $k_0>0$ such that its quasi-isometry type is represented by any subspace $X^G_k$, for $k \geq k_0$. In particular, the $G$-action on $X^G_k$ is coarsely ineffective (see Section \ref{sec: four}), hence 
 the projection map $\pi\colon X \to X/G$  induces a coarse equivalence   $X^G_k \cong \pi(X^G_k)$. We will use this observation to consider $X^G_{bd}$ also as a subspace of $X/G$.
\end{remark}

For completeness, we will describe the transfer map for coarse homology.
Let $\{\cU_n\}$ be a regular $G$-coarsening system of $X$ and $K(\cU_n)$
the nerve of $\cU_n$. Then we have the transfer map for locally finite homology (compare \cite[III.2]{bredon1}):
$$\mu_* \colon \HLF_*\bigl(K(\cU_n)/G\bigr)\to \HLF_*\bigl(K(\cU_n)\bigr)$$
and, for $m\ge n$ we have the induced commutative diagram
$$\xymatrix{
{\HLF\bigl(K(\cU_n)/G\bigr)} \ar[r]^-{\mu_*} \ar[d] &
{\HLF\bigl(K(\cU_n)\bigr)} \ar[d] \\
{\HLF\bigl(K(\cU_m)/G\bigr)}  \ar[r]^-{\mu_*} &
{\HLF\bigl(K(\cU_m)\bigr)} } $$
by naturality of the transfer.

Passing to the direct limit, and applying Corollary \ref{orbit space}, we obtain the transfer for coarse homology
$$\mu_*\colon \ch_*(X/G)\to\ch_*(X)$$
which satisfies
\eqncount
\begin{equation}\label{transfer property}\begin{split}
\pi_*\mu_*=|G|\colon \ch_*(X/G)\to\ch_*(X/G) \\
\mu_*\pi_*=\sigma_*=\sum_{g\in G} g_*\colon \ch_*(X)\to\ch_*(X)
\end{split}\end{equation}
where $\pi\colon X\to X/G$ is the canonical projection. 

We can define coarse homology for other coefficient groups $\Lambda$, just by tensoring the simplicial chain complexes of the $K(\cU_n)$ with $\Lambda$ (over $\bZ$), and passing to the direct limit of the locally-finite homology. The  formulas above continue to hold, and we have the usual consequence:
\begin{proposition}[{\cite[III. 2.4]{bredon1}}]
If $\Lambda$ is a field of characteristic 0 or prime to $|G|$, then
$$\pi_*\colon \ch_*(X;\Lambda)^G\to\ch_*(X/G;\Lambda)$$ is an isomorphism, as is
$$\mu_*\colon \ch_*(X/G;\Lambda)\to\ch_*(X;\Lambda)^G.$$
\end{proposition}
In the statement, the coefficients $\Lambda$ have trivial $G$-action, and 
$\ch_*(X;\Lambda)^G$ denotes the fixed set of the induced $G$-action 
on the coarse homology (see Bredon \cite[III.2]{bredon1} for the usual transfer and its properties). The formulas \eqref{transfer property} can be used to prove a coarse version of a result of Floyd \cite[III.5.4]{bredon1}.

\section{P.~A.~Smith Theory}
The goal of P.~A.~Smith theory \cite{pasmith1} is to relate the mod $p$ homology of a regular simplicial $G$-complex $K$ and its fixed set $K^G$, in the case when $G$ is a finite $p$-group and $p$ is a prime (see the Borel Seminar \cite{borel-seminar}).  We will
follow the procedure from Bredon \cite[Chap.~III]{bredon1} to pass from a
simplicial complex to a coarse space, and establish the analogue of the  P. A. Smith theory for coarse homology.

We first review the classical P.~A.~Smith theory. Let $p$ be a prime, and let $G = \cy p$,  denote the cyclic group of order $p$. Fix a generator $g \in G$, and define elements
$$ \sigma = 1 + g + g^2 + \dots + g^{p-1} \qquad \text{and} \qquad \tau = 1-g$$
in the group ring $\Fp G$. If $\varrho = \tau^j$, for $1 \leq j \leq p-1$, then we define
$\bar\varrho = \tau^{p-j}$. Since $g^p = 1$, we have $\sigma \tau = \tau \sigma = 0$ and $\sigma = \tau^{p-1}$. 

The elements $\sigma$ and $\tau$, as well as $\varrho$ and $\bar\varrho$, are considered as operators on the mod~$p$ homology $H_*(K; \Fp)$ of a regular simplicial $G$-complex. In the rest of the section, all homology groups will be understood with $\Fp$-coefficients.

Let $K$ denote a regular simplicial $G$-complex, and $L 
\subseteq K$ a $G$-invariant subcomplex. The simplicial chain complex of the pair $(K, L)$ is denoted $C(K,L)$. The  \emph{Smith special homology groups} 
$$H^\varrho_*(K, L; \Fp) = H_*(\varrho C(K,L); \Fp)$$
are defined for each $\varrho = \tau^j$, $1 \leq j \leq p-1$, as the mod $p$ homology of the chain complex $\varrho C(K,L) \subset C(K,L)$.
The main results of the classical P.~A.~Smith theory  are based on the exact sequences of chain complexes:
\eqncount
\begin{equation}\label{smith chain complexes}
\begin{split}
0 \to \bar\varrho C(K,L) \oplus C(K^G, L^G) \xrightarrow{i} C(K,L) \xrightarrow{\varrho} \varrho C(K, L) \to 0\\
0 \to \sigma C(K, L) \to \tau^j C(K, L) \xrightarrow{\tau^{j+1}} C(K, L) \to 0
\end{split}
\end{equation}
valid for $\rho = \tau^j$ and $1\leq j \leq p-1$. To generalize Smith theory to our setting, we first apply locally finite homology to these chain complexes.

\begin{definition} Let $K$ denote a regular simplicial $G$-complex, and $L 
\subset K$ a $G$-invariant subcomplex. Then  the  \emph{locally finite} Smith special homology groups
$$\Hlf{\varrho}_*(K, L) := \HLF_*(\varrho C(K, L);\Fp)$$
are defined,  for
 $\rho = \tau^j$ and $1\leq j \leq p-1$.
\end{definition}
The Smith special homology group for $\sigma =\tau^{p-1}$ is related to the orbit spaces (denoted $K^\bigast = K/G$ and $L^\bigast = L/G$), via the formula
\eqncount
\begin{equation}\label{orbit space iso}
\Hlf{\sigma}_*(K, L;\Fp) \cong \HLF_*(K^\bigast, K^G \cup L^\bigast)
\end{equation}
We have the following exact triangle (locally finite homology with $\Fp$ coefficients).
\begin{theorem}\label{exact triangle} 
For $\varrho = \tau^j$, $1 \leq j \leq p-1$, there is an exact triangle 
$$\xymatrix{
 & {\HLF_*(K,L)} \ar[dl]_{\varrho_*}& \\
{\Hlf{\varrho}_*(K,L)} \ar[rr]^(0.4){\delta_*} & & {\Hlf{\bar{\varrho}}_*(K,L)\oplus \HLF_*(K^G, L^G)} \ar[ul]_{i_*} } $$
where the horizontal map $\delta_*$ has degree $-1$ and the other maps $i_*$, $\varrho_*$ have degree $0$. 
\end{theorem}
\begin{proof}
See Bredon \cite[III.3.3]{bredon1}. The map $i_*$ is induced from the direct sum of the inclusions of the subcomplexes $\bar\varrho C(K, L)$ and $C_*(K^G, L^G)$ in $C(K, L)$ in \eqref{smith chain complexes}.
\end{proof}

 The first application of this exact triangle is to establish the``Smith inequalities" for locally finite homology.
\begin{theorem}\label{ineg}
Let $K$ be a finite-dimensional regular $G$-complex, and $L 
\subseteq K$ a $G$-invariant subcomplex. Then, for any $n\ge 0$ and
for any $\varrho=\tau^j$ with $1\le j\le p-1$,
$$\rk  \Hlf{\varrho}_n(K, L) +\sum_{i\ge n}\rk   \HLF_i(K^G, L^G)
\le\sum_{i\ge n}\rk  \HLF_i(K, L).$$ If the right-hand side is finite, then
the left-hand side is also finite, and in particular $\rk  \HLF_i(K^\bigast, K^G\cup L^\bigast)$ is  finite, for
all $i\ge n$.
\end{theorem}
\begin{proof}
See Bredon \cite[III.4.1]{bredon1} or Floyd \cite{floyd2}. By ``rank" we mean the rank (or dimension) of the indicated homology groups as $\Fp$-vector spaces. The last part follows from \eqref{orbit space iso}.
\end{proof}

We now pass from locally finite homology to coarse homology when $G=\cy p$ acts by isometries on a proper metric space. Our results will apply equally to coarse $G$-spaces: we again use Theorem \ref{kleiner-leeb} to justify replacing a coarse $G$-action by a coarsely equivalent isometric $G$-action. 
\begin{definition} Let $X$ be a proper metric space with a coarse action of a  finite group $G$. We say that $X$ is $G$-\emph{finitistic} if the action admits a regular $G$-coarsening system $\{\cU_n\}$ whose nerves $K(\cU_n)$ have uniformly bounded dimension, for all $n$.
\end{definition}

%%%%%%%%%%%%%%%%%
For the remainder of this
section,  $X$ is a proper metric space and $G = \cy{p}$ acts by isometries on $X$. We will also assume that $X$ is $G$-finitistic and that 
 the bounded fixed set $X^G_{bd}$ exists, or the $G$-action is \emph{tame} (see Definition \ref{bdd}).

For any regular $G$-coarsening system $\{\cU_n\}$ associated to the $G$-action on $X$, we have seen that any two equivariant
coarsening maps induce contiguous equivariant simplicial maps and the induced chain
maps are equivariantly chain homotopic. Thus we can define
  the \emph{Smith special homology groups} for coarse homology
$$\ch^\varrho_*(X)=\varinjlim\Hlf{\varrho}_*\bigl(K(\cU_n)\bigr),$$
for $\varrho = \tau^j$ and $1\leq j \leq p-1$. The Smith homology groups 
 are natural with respect to equivariant coarse maps.

Passing to the limit in the Smith triangle (see Theorem \ref{exact triangle}) associated to
$\Hlf{\varrho}_*\bigl(K(\cU_n)\bigr)$, and by using Corollary \ref{rhomap}, we obtain
the Smith exact triangle for coarse homology:
\eqncount
\begin{equation}\label{coarse exact triangle}
\vcenter{\xymatrix{
 & {\ch_*(X)} \ar[dl]_{\varrho_*}& \\
{\ch^\varrho_*(X)} \ar[rr]^(0.4){\delta_*} & & {\ch^{\bar{\varrho}}_*(X)\oplus\ch_*(X^G_{bd})} \ar[ul]_{i_*} } }
\end{equation}
where the horizontal map has degree $-1$ and the other maps have degree $0$. This triangle
is exact because the direct limit functor is exact.
Note that the isomorphism from \eqref{orbit space iso}
$$\Hlf{\sigma}_*\bigl(K(\cU_n)\bigr)\cong \HLF_*\bigl(K(\cU_n)/G,
K(\cU_n)^G\bigr)$$
is natural, and thus
\eqncount
\begin{equation}\label{rel.hom}
\ch^\sigma_*(X)\cong\ch_*(X/G,X^G_{bd}),
\end{equation}
where $X^G_{bd}$ is considered via the projection map as a subspace of $X/G$ (see Remark \ref{rem: orbit}).

We now establish the coarse version of the Smith inequalities.
\begin{theorem}\label{thm: pasmith}
Let $X$ be a proper metric space and
$G$ a cyclic group of prime order $p$, with a coarse action  on $X$. Assume that $X$ is $G$-finitistic and that $X^G_{bd}$ exists. Then, for any $n\ge 0$ and
for any $\varrho=\tau^j$ with $1\le j\le p-1$,
$$\rk \ch_m^\varrho(X) +\sum_{i\ge m}\rk  \ch_i(X^G_{bd})
\le\sum_{i\ge m}\rk  \ch_i(X).$$
If the right-hand side is finite, then the left-hand side is also finite, 
and in particular $\rk  \ch_i(X/G, X^G_{bd})$ is finite, for all $i\ge m$.
\end{theorem}

\begin{proof}
Since $X$ is $G$-finitistic, there is a regular $G$-coarsening system $\{\cU_n\}$ for $X$ and an integer $q$, such that $\dim K(\cU_n)\leq q$ for all $n$.
Thus $\ch_j(X)=0$ for $j>q$ and,
similarly $\ch^\varrho_j(X)=0$  since the nerves $K(\cU_n)$ do not contain any
$q+1$ simplices, for any $n$. Let
\begin{align*}
a^n_i&= \rk  \HLF_i\bigl(K(\cU_n)^G\bigr) &
b^n_i&= \rk  \HLF_i\bigl(K(\cU_n)\bigr) \\
c^n_i&= \rk  \Hlf{\varrho}_i\bigl(K(\cU_n)\bigr) &
\bar{c}_i&= \rk  \Hlf{\bar{\varrho}}_i\bigl(K(\cU_n)\bigr).
\end{align*}
From Theorem \ref{ineg} it follows that, for any $n$,
$$c^n_m +\sum_{i\ge m}a^n_i\le\sum_{i\ge m}b^n_i .$$
The above discussion shows that $b^n_i=0$ for $i>q$ and for any $n$. Passing to the
limit over $n$ we obtain
$$\rk  \ch_m^\varrho(X) +\sum_{i\ge m}\rk  \ch_i(X^G_{bd})
\le\sum_{i\ge m}\rk  \ch_i(X).$$
The last part follows from equation \eqref{rel.hom}.
\end{proof}

In the next statement, the notation $\chi(X)$, $\chi(X^G_{bd})$, and $\chi(X/G)$ means the Euler characteristic with respect to their coarse homology, and $\rk \ch (X)$ denotes its ``total rank" (the sum of the ranks over $\Fp$ of all the coarse homology groups of $X$).
\begin{theorem}\label{thm: coarse euler}
Let $X$ and $G$ be as before with $\rk \ch_*(X)<\infty$. Then
$$\chi(X)+(p-1)\chi(X^G_{bd})=p\chi(X/G),$$ therefore $\chi(X^G_{bd})\equiv\chi(X)\pmod p$.
\end{theorem}

\begin{proof}
Compare Bredon \cite[III.4.3]{bredon1} or Floyd \cite{floyd2}. The exact sequence of the inclusion $X^G_{bd} \subseteq X/G$ and Theorem \ref{thm: pasmith} shows that
$\rk \ch(X/G) < \infty$, and hence all three Euler characteristics are defined. We have the relation
$$\chi(X/G) = \chi(X/G, X^G_{bd}) + \chi(X^G_{bd})$$
by considering the long exact sequence of the pair $(X/G, X^G_{bd})$ are a chain complex with zero homology. Let $\chi(\varrho) := \chi(\ch^\varrho(X))$, and note that
$$\chi(X) = \chi(\sigma) + \chi(\tau) + \chi(X^G_{bd})$$
from the Smith exact triangle \eqref{coarse exact triangle} for $\varrho = \sigma$. Now 
from the second exact sequence of Smith chain complexes in \eqref{smith chain complexes} one obtains the equations
$\chi(\tau^j) = \chi(\tau^{j+1}) + \chi(\sigma)$, for $1\leq j \leq p-2$. By adding all these equations, and using $\chi(\tau^{p-1}) = \chi(\sigma)$, we get
$$\chi(X) = p\chi(\sigma) + \chi(X^G_{bd}).$$
But $\chi(\sigma) = \chi(X/G, X^G_{bd})$ by \eqref{rel.hom}, hence
$$\chi(X) = p(\chi(X/G) - \chi(X^G_{bd})) + \chi(X^G_{bd})$$
which gives the required formula.
\end{proof}

\section{The Proof of Theorem A}\label{proof of thma}
The most often used result of the classical P.~A.~Smith theory is the application to actions on mod $p$ homology spheres. For our coarse version of this result, we observe that Euclidean space $\bbR^m$ is a coarse $m$-sphere.  We use this
terminology because the coarse homology of $\bbR^m$ is equal to the reduced
ordinary homology of an $m$-sphere, by \eqref{euclidean space}. More generally, a (mod $p$) \emph{coarse $m$-sphere}
 is a metric space with  the same (mod $p$) coarse homology as $\bbR^m$. We have the following  application of coarse P.~A.~Smith theory.
\begin{theorem}\label{final}
Let  $X$ be a proper
 metric space, which is a (mod $p$) coarse
homology $m$-sphere, for some prime $p$. Let $G$ be a finite  $p$-group with a  coarse action on $X$. Assume that $X$ is $G$-finitistic and that the $G$-action is tame. Then $X^G_{bd}$ is
a (mod $p$)  coarse homology $r$-sphere, for some $0\le r\le m$. If $p$ is odd, then $m-r$ is even.
\end{theorem}

\begin{proof}
By Theorem \ref{kleiner-leeb} we may assume that $G$ acts by isometries on $X$. Since $X$ is $G$-finitistic,  there is a regular $G$-coarsening system $\{\cU_n\}$ for $X$ and an integer $q$, such that $\dim K(\cU_n)\leq q$ for all $n$. Let $G_0\triangleleft G$ be a normal subgroup of index $p$ in $G$, so that
$G/G_0 \cong \cy p$. By induction on the order of $G$, we may assume that $X^{G_0}_{bd}$ is a (mod $p$) coarse homology $t$-sphere, for some $0\leq t \leq m$, and $m-t\equiv 0 \pmod 2$  if $p$ is odd. Since the $G$-action on $X$ is tame, there exists an $k_0 >0$ such that
 $X^G_k$ is coarsely equivalent to $X^G_{bd}$, and  $X^{G_0}_k$ is coarsely equivalent to $X^{G_0}_{bd}$, for any $k > k_0$ (see Definition \ref{bdd}). 
 
 Fix $k > k_0$, and let $\ZZ:= X^{G_0}_k$. Since $\ZZ \subset X$ is a $G$-invariant subspace, and the $G_0$-action on $\ZZ$ is coarsely ineffective, this subspace $\ZZ$ inherits a coarse $G/G_0$-action from the $G$-action on $X$.
The inductive step will be completed  by applying Theorem  \ref{thm: pasmith}, to  $\ZZ$, after we check that the hypotheses are satisfied.

\medskip
\noindent
(i) 
\emph{The space $\ZZ=X^{G_0}_k$ is $G/G_0$-finitistic}.  In Remark \ref{rem: fixed covering}, we noted that there is a cofinal subsequence $\{\cU_{m_i}\}$ of the regular $G$-coarsening system such that
$\{\cU^{G_0}_{m_i}\}$ is a regular $G$-coarsening system for the subspace
$X^{G_0}_k$. Since $G_0$ acts trivially on the nerves $K(\cU^{G_0}_{m_i})$, we have a regular $G/G_0$-coarsening system for $X^{G_0}_k$.
 
 \smallskip
 \noindent
(ii) 
\emph{The coarse $G/G_0$-action on $\ZZ=X^{G_0}_k$ is tame}.
 Since 
 $$X^G_k = (X^{G_0}_k)^G_k \subseteq (X^{G_0}_k)^G_l \subseteq
 (X^{G_0}_l)^G_l = X^G_l$$
 for any $l \geq k > k_0$, and the inclusion $X^G_k \subset X^G_l$ is coarsely dense,  the inclusion 
 $$\ZZ^{G/G_0}_k = (X^{G_0}_k)^G_k \subseteq (X^{G_0}_k)^G_l = \ZZ^{G/G_0}_l$$
 is also coarsely dense. Hence the
 coarse $G/G_0$-action on $\ZZ$ is tame, and $\ZZ^{G/G_0}_k = \ZZ^{G/G_0}_{bd}$.

  We now apply Theorem
 \ref{thm: pasmith} to the (mod $p$) coarse homology $t$-sphere $\ZZ=X^{G_0}_k$, with respect to the induced coarse action of $G/G_0 \cong \cy p$.
 We have $\rk \ch_*(\ZZ) = 1$, which implies $\rk \ch_*(\ZZ^{G/G_0}_{bd}) \leq 1$. But $\rk \ch_*(\ZZ^{G/G_0}_{bd}) =0$ is not possible, since $\chi(\ZZ^{G/G_0}_{bd} )\equiv\chi(\ZZ)\pmod p$, by
 Theorem \ref{thm: coarse euler}. Therefore 
 $$\rk \ch_*(\ZZ^{G/G_0}_{bd}) = 1, $$
  so that $\ZZ^{G/G_0}_{bd}$ is a coarse homology $r$-sphere, for some $0 \leq r \leq t$. If $p$ is odd, then $\chi(\ZZ^{G/G_0}_{bd})\equiv\chi(\ZZ)\pmod p$ implies that $t$ and $r$
are either both odd or both even, thus $t-r$ is even. Since 
$$\ZZ^{G/G_0}_{bd} = (X^{G_0}_k)^{G/G_0}_k = (X^{G_0}_k)^{G}_k = X^G_k = X^G_{bd}$$
for $k > k_0$, and $0 \leq r \leq t \leq m$ with $m-r \equiv 0\pmod 2$ if $p$ odd, we are done.
\end{proof}

\begin{proof}[The proof of Theorem A] Here $X$ is a proper geodesic metric space, with finite asymptotic dimension, and $G$ is a finite $p$-group acting by isometries. We can apply Theorem 
\ref{regular coarsening system} to conclude that $X$ is $G$-finitistic.  Now Theorem A follows from Theorem \ref{final}.
\end{proof}

We conclude by giving a coarse version of another well-known application of classical Smith theory:  the group $G = \cy p \times \cy p$, for $p$ a prime, can not act freely on a finitistic mod $p$ homology $m$-sphere (see Bredon \cite[III.8.1]{bredon1}). 

The bounded fixed set of a coarse finite group action is never empty (unlike the actual fixed set), so we need to define a coarse version of the term ``free action". 
\begin{definition} A tame $G$-action on a proper metric space $X$ is called  \emph{semifree at the large scale} if  $X^H_{bd} = X^G_{bd}$ for all non-trivial subgroups
$\{e\}\neq H\leq G$. A coarsely effective
 action is called \emph{free at the large scale} if there exists a compact subset $Y\subset X$, such that $X^H_{bd} = Y$, for all non-trivial subgroups
$\{e\}\neq H\leq G$.
\end{definition} 

\begin{example}
Let $X = \bbR^2$ with $G = \cy p$, for $p$ prime, acting by rotations. We call this a ``free" action in the coarse sense because the fixed set, although non-empty, is compact. The action is free at the large scale and $X^G_{bd} = \{0\}$.
\end{example}

\begin{theorem}\label{thm: semifree}
The group $G = \cy p \times \cy p$, for $p$ a prime, can not act tamely and semifreely at the large scale on a (mod $p$) coarse
homology $m$-sphere $X$, whenever
 $X$ is $G$-finitistic, and
$X^G_{bd}$ is
a (mod $p$)  coarse homology $r$-sphere, for some $0\le r< m$.
%\begin{enumerate}
%\item $X$ is $G$-finitistic, and
%\item $X^G_{bd}$ is
%a (mod $p$)  coarse homology $r$-sphere, for some $0\le r< m$.
%\end{enumerate}
\end{theorem}
In particular, under the above assumptions, $G = \cy p \times \cy p$, for $p$ a prime, can not act freely at the large scale on a (mod $p$) coarse
homology $m$-sphere, if $m>0$.
\begin{proof}[The proof of Theorem B]  We again apply
Theorem \ref{regular coarsening system} to conclude that $X$ is $G$-finitistic. Then Theorem B follows from Theorem \ref{thm: semifree}.
\end{proof}
\begin{proof}[The proof of Theorem \ref{thm: semifree}]
We write $G = G_1 \times G_2$, where $G_1$ and $G_2$ are cyclic subgroups of order $p$. Suppose, if possible, that $G$ acts semifreely on $X$ under the assumptions given in the statement. Then $X^{G_1}_{bd}= X^G_{bd}$ is a (mod $p$)  coarse homology $r$-sphere, for some $0\le r< m$. It will be convenient to replace $X$ by a coarsely equivalent space $Z = X\times E$, where $E=S^k\times S^k$  for some $k > 2m$, $k$ odd. In this space, we have a product action of $G$, where $G= G_1 \times G_2$ acts freely on $E=S^k\times S^k$ via  free isometric actions of $G_1$, $G_2$ on each sphere factor. The projection map $Z = X \times E \to X$ is a $G$-equivariant coarse equivalence, so $Z^G_{bd} = X^G_{bd}$.

We first observe that $\rk HC_{m+1}(Z/G, Z^G) > 0$. To compute this coarse homology group we use a regular $G$-coarsening system $\{\cU_n\}$ for $X$, and form the coarsening system $\{\cV_n\}$ for $Z$, with $\cV_n = \{ U \times E\vv U \in\cU_n\}$ for each $n$. Then 
$$HC_*(Z/G, Z^G) = \varinjlim \HLF_*(K(\cV_n)/G, K(\cV_n)^G).$$
To compute the right-hand side, we use the direct limit of the spectral sequences of the fibrations
$K(\cV_n) \to K(\cV_n)/G \to BG$ in locally finite homology with $\Fp$-coefficients, with $E^2_{s,q} = H_s(BG; \HLF_q(K(\cV_n),  K(\cV_n)^G))$. We obtain a spectral sequence with 
$$E^2_{s,q} =H_s(BG; \ch_q(Z. Z^G_{bd})) \Rightarrow \ch_*(Z/G, Z^G_{bd})$$ converging to $\ch_*(Z/G, Z^G_{bd})$.
Since $r <m$, the only possible differential $$d_{m-r}: H_{m-r+s}(BG; \ch_{r+1}(Z, Z^G_{bd})) \to H_s(BG; \ch_{m}(Z, Z^G_{bd}))$$ can not be injective, since
for example, $\ch_{r+1}(Z, Z^G_{bd})) = \Fp = \ch_{m}(Z, Z^G_{bd})$, and $\rk H_{m-r}(BG; \Fp) \geq 2$.

We can consider $Z/G$ as the quotient of $Z/G_1$ by the remaining $G_2$-action. By Theorem \ref{thm: pasmith} and formula \eqref{rel.hom} applied to $Z$ with the $G_1$-action, we have  $\ch_{m+1}(Z/G_1, Z^G_{bd}))=0$. By the same results applied to the $G_2$-action on $Z/G_1$, we have the inequality
$$\rk \ch_{m+1}(Z/G, Z^G_{bd})) = \rk \ch^\sigma_{m+1}(Z/G_1, Z^G_{bd}) \leq \rk \ch_{m+1}(Z/G_1, Z^G_{bd})),$$ which contradicts the calculation above.
\end{proof}

\begin{acknowledgement} The first author would like to thank Erik Kj\ae r Pedersen for helpful conversations about coarse geometry which led to the notion of the bounded fixed set, and to the second author's Ph.D thesis topic \cite{savin1} on which this paper is based.
\end{acknowledgement}

%\bibliographystyle{ih}
%\bibliography{ihmain}
%\end{document}
%%%%%%%%%%%%%%%%%%%%%%%%
\providecommand{\bysame}{\leavevmode\hbox to3em{\hrulefill}\thinspace}
\providecommand{\MR}{\relax\ifhmode\unskip\space\fi MR }
\providecommand{\MRhref}[2]{%
  \href{http://www.ams.org/mathscinet-getitem?mr=#1}{#2}
}
\providecommand{\href}[2]{#2}

\end{document}